\documentclass{amsart}

\usepackage{amssymb,amsmath,amsthm,amsfonts,mathtools}

\usepackage[
paperwidth = 8.5in,
paperheight = 11in,
left   = 1in,
right  = 1in,
top    = 1in,
bottom = 1in
]{geometry}

\usepackage{graphicx}
\usepackage{enumitem}
\usepackage{lmodern} 
\usepackage[font=small, labelfont=bf]{caption}
\usepackage[font=small, labelfont=bf]{subcaption}
\usepackage{booktabs}
\usepackage{xcolor}
\usepackage{tikz}
\usetikzlibrary{arrows.meta,positioning,calc,decorations.pathreplacing,backgrounds,fit,shapes.geometric}
\usepackage{pgfplots}
\pgfplotsset{compat=1.18}

\definecolor{royalpurple}{RGB}{99,6,161}
\definecolor{todoblue}{RGB}{0,70,180}
\definecolor{commentpurple}{RGB}{115,0,150}
\definecolor{commentorange}{RGB}{210,100,0}
\definecolor{suggestgreen}{RGB}{0,120,70}
\definecolor{aside}{RGB}{30,90,170}
\definecolor{bside}{RGB}{0,140,90}
\definecolor{shortc}{RGB}{20,110,190}
\definecolor{longc}{RGB}{210,80,15}
\definecolor{cutred}{RGB}{200,40,40}
\definecolor{corecol}{RGB}{30,90,170}
\definecolor{tilecol}{RGB}{0,120,80}
\definecolor{Ccol}{RGB}{170,80,0}
\definecolor{specblue}{RGB}{30,90,170}
\definecolor{matgreen}{RGB}{0,120,70}

\newtheorem{theorem}{Theorem}[section]
\newtheorem{lemma}[theorem]{Lemma}
\newtheorem{proposition}[theorem]{Proposition}
\newtheorem{corollary}[theorem]{Corollary}

\newtheorem{theoremletter}{Theorem}

\newtheorem*{theoremletterprime}{Theorem B$'$}

\theoremstyle{definition}
\newtheorem{definition}[theorem]{Definition}
\newtheorem{example}[theorem]{Example}

\theoremstyle{remark}
\newtheorem{remark}[theorem]{Remark}

\usepackage{algorithm}
\usepackage{algpseudocode}

\usepackage{upquote}

\usepackage{hyperref}
\hypersetup{
colorlinks=true,
linkcolor=blue,
citecolor=royalpurple,
urlcolor=blue
}

\newcommand{\RS}{\mathcal{R}}
\newcommand{\Aut}{\operatorname{Aut}}
\newcommand{\Exp}{\mathbb{E}}

\newcommand{\NB}{\mathrm{NB}}
\newcommand{\gmax}{g_{\max}}

\newcommand{\cyc}{\mathrm{cyc}}
\newcommand{\nonsimp}{\mathrm{nonsimp}}

\newcommand{\Sym}{\operatorname{Sym}}
\newcommand{\Gtau}{G^{\tau}}
\newcommand{\du}{\mathbin{\dot{\cup}}}
\newcommand{\tr}{\operatorname{tr}}

\numberwithin{equation}{section}

\begin{document}

\title{On the Genus Polynomial of Cubic Graphs}

\author[A. Ulrigg]{Austin Ulrigg}
\address{University of Washington}
\email{austinul@uw.edu}


\begin{abstract}
The orientable genus polynomial of a graph counts its cellular embeddings by genus. For finite simple $2$-connected cubic graphs it is a cycle-matroid invariant: $M(G)\cong M(H)$ implies $\Gamma_G=\Gamma_H$. The adjacency spectrum and the genus polynomial are incomparable: neither determines the other. We exhibit connected cubic graphs on $16$ vertices sharing the adjacency spectrum, spanning-tree count, girth, diameter, vertex and edge connectivity, automorphism-group order, and cycle counts through length $10$, yet with pairwise distinct genus polynomials. Splitting the expected face count at twice the girth explains the difference: short faces are spectral, long faces are not. We construct an explicit infinite family of connected cospectral cubic pairs $(G_t,H_t)$ on $14+2t$ vertices whose minimum genera differ. We also compute the genus polynomials of all $7{,}875{,}918$ connected cubic graphs through $22$ vertices and derive from short-cycle counts a deterministic lower bound on the minimum genus.
\end{abstract}
\maketitle

\section{Introduction}

The \emph{orientable genus} of a graph is the least number of handles that must be added to the sphere before the graph can be drawn on the resulting surface without crossings. Computing the orientable genus of a graph is hard: deciding whether a graph has genus at most $k$ is NP-complete~\cite{Thomassen}, although for each fixed surface there is a linear-time embedding algorithm~\cite{Mohar1999}. Every cellular embedding of a connected graph in an orientable surface is described, uniquely, by a cyclic ordering of the edges around each vertex called its \emph{rotation system}, and every rotation system describes a 2-cell embedding~\cite{Edmonds,GrossTucker,MoharThomassen}. Cellular embeddings are therefore finite combinatorial objects, and questions about embeddings can be described equivalently as questions about rotation systems.

For a connected cubic graph $G$ on $n$ vertices, each vertex has exactly two cyclic orders, so $G$ has precisely $2^n$ rotation systems. The genus of the embedding determined by a rotation system $\rho$ depends only on its face count $F(\rho)$, through Euler's formula:
\begin{equation}\label{eq:intro-genus-from-faces}
g(\rho)=1+\frac n4-\frac12\,F(\rho),
\qquad\text{so}\qquad
\gamma(G)=1+\frac n4-\frac12\max_\rho F(\rho).
\end{equation}
The minimum genus is one extreme of the distribution which we record by the \emph{orientable genus polynomial}
\begin{equation}\label{eq:intro-genus-poly}
\Gamma_G(x)=\sum_{g\ge 0} N_g(G)\,x^{g},
\end{equation}
where $N_g(G)$ is the number of rotation systems of $G$ whose embedding has genus $g$ (Figure~\ref{fig:genus-surfaces}). Its least exponent is the minimum genus $\gamma(G)$. Its greatest is the maximum genus $\gmax(G)$, which is computable by Xuong's theorem~\cite{Xuong}. By Duke's interpolation theorem~\cite{Duke}, every orientable genus in between is realized by at least one rotation system.

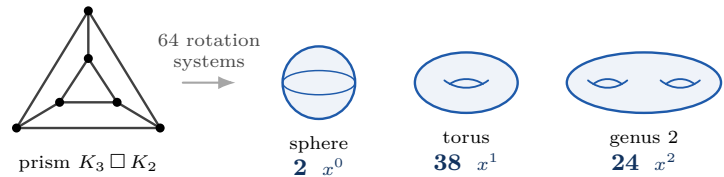
\begin{figure}[H]
\centering
\begin{tikzpicture}[line cap=round, line join=round, >=Latex,
   dot/.style={circle, fill=black, inner sep=1.2pt},
   surf/.style={specblue, line width=0.9pt},
   sfill/.style={fill=specblue!7}]
\coordinate (O1) at (0,1.05); \coordinate (O2) at (-0.95,-0.5); \coordinate (O3) at (0.95,-0.5);
\coordinate (I1) at (0,0.42);  \coordinate (I2) at (-0.38,-0.16); \coordinate (I3) at (0.38,-0.16);
\draw[black!75,line width=0.9pt] (O1)--(O2)--(O3)--cycle;
\draw[black!75,line width=0.9pt] (I1)--(I2)--(I3)--cycle;
\draw[black!75,line width=0.9pt] (O1)--(I1) (O2)--(I2) (O3)--(I3);
\foreach \p in {O1,O2,O3,I1,I2,I3}{\node[dot] at (\p){};}
\node[font=\scriptsize,anchor=north] at (0,-0.72) {prism $K_3\,\square\,K_2$};
\draw[->,gray!70,line width=0.8pt] (1.25,0.1) -- (1.95,0.1);
\node[font=\scriptsize,align=center,gray!55!black] at (1.6,0.5) {$64$ rotation\\ systems};
\begin{scope}[shift={(3.05,0.1)}]
  \draw[surf,sfill] (0,0) circle (0.48);
  \draw[surf,line width=0.55pt] (0,0) ellipse (0.48 and 0.16);
  \node[font=\scriptsize,anchor=north] at (0,-0.58) {sphere};
  \node[font=\small,specblue!55!black,anchor=north] at (0,-0.88) {$\mathbf{2}$ \;{\scriptsize$x^{0}$}};
\end{scope}
\begin{scope}[shift={(5.0,0.1)}]
  \draw[surf,sfill] (0,0) ellipse (0.68 and 0.4);
  \draw[surf,line width=0.7pt] (-0.29,0.03) .. controls (-0.12,-0.10) and (0.12,-0.10) .. (0.29,0.03);
  \draw[surf,line width=0.7pt] (-0.22,-0.02) .. controls (-0.10,0.07) and (0.10,0.07) .. (0.22,-0.02);
  \node[font=\scriptsize,anchor=north] at (0,-0.52) {torus};
  \node[font=\small,specblue!55!black,anchor=north] at (0,-0.82) {$\mathbf{38}$ \;{\scriptsize$x^{1}$}};
\end{scope}
\begin{scope}[shift={(7.35,0.1)}]
  \draw[surf,sfill] (0,0) ellipse (1.02 and 0.4);
  \foreach \cx in {-0.48,0.48}{
    \draw[surf,line width=0.7pt] (\cx-0.26,0.03) .. controls (\cx-0.10,-0.10) and (\cx+0.10,-0.10) .. (\cx+0.26,0.03);
    \draw[surf,line width=0.7pt] (\cx-0.19,-0.02) .. controls (\cx-0.08,0.07) and (\cx+0.08,0.07) .. (\cx+0.19,-0.02);}
  \node[font=\scriptsize,anchor=north] at (0,-0.52) {genus $2$};
  \node[font=\small,specblue!55!black,anchor=north] at (0,-0.82) {$\mathbf{24}$ \;{\scriptsize$x^{2}$}};
\end{scope}
\end{tikzpicture}
\caption{The triangular prism $K_3\,\square\,K_2$ has $\Gamma=2+38x+24x^{2}$.}
\label{fig:genus-surfaces}
\end{figure}

The mean of the same distribution, the \emph{average genus}, lies well above the minimum: among $3$-connected graphs only $K_4$ has average genus below $1$~\cite{GrossKleinRieper}, and most of a genus distribution typically lies in the interior of its range, as it does for bouquets~\cite{Stahl}. A uniformly random rotation system is usually a poor estimate of $\gamma(G)$~\cite{ChenGao}. The genus distribution, therefore, contains strictly more topological information than the minimum genus alone. The orientable genus polynomial (which we interchangeably call the \emph{genus distribution}) has been studied since Gross and Furst placed the embedding-distribution invariants in a hierarchy~\cite{GrossFurst}, and it has been computed for many families, including bouquets of circles, infinite classes~\cite{GrossRobbinsTucker,FurstGrossStatman}, and series--parallel and cubic graphs~\cite{ChenGrossRieper,GrossKotrbcik}. Two conjectures motivated much of this development and have since been refuted: Stahl's conjecture that \emph{every genus polynomial is real-rooted}~\cite{StahlZeros} was disproved by Chen and Liu~\cite{ChenLiu}, and cubic genus polynomials with non-real roots have been discovered by Carr, Dhaliwal, and Mohar~\cite{CarrDhaliwalMohar}. The Gross--Robbins--Tucker \emph{log-concavity conjecture}~\cite{GrossRobbinsTucker} was disproved in general by Mohar~\cite{MoharLogConvex}. Our census relates to the first refuted conjecture. Among connected cubic graphs on $10,12,14,16,18$ vertices, exactly $1,5,26,194,2022$ have a genus polynomial with a non-real root, and at $n=20$ and $n=22$ the counts are $33{,}890$ and $578{,}044$ (Section~\ref{sec:census}). Our census reproduces the polynomials of~\cite{CarrDhaliwalMohar} exactly and corrects the non-real-root totals reported there ($2,41,178$ at $n=10,14,16$). The authors have confirmed the revised counts and posted an updated version (M.~Carr, personal correspondence, July 2026).

Matroids were introduced by Whitney to study what the linearly independent
subsets of the columns of a matrix and the acyclic subsets of the edges of a
graph have in common~\cite{Whitney35}. For a graph, the canonical matroid is called the \emph{cycle matroid}, $M(G)$. Its ground set is $E(G)$, and a set of edges is independent exactly when it contains no cycle. Distinct graphs can share a cycle matroid, and Whitney's $2$-isomorphism theorem classifies exactly which do~\cite{Whitney}. Another graph polynomial, the Tutte polynomial, is already known to be a cycle-matroid invariant~\cite{Oxley}. One can naturally ask whether the genus polynomial is also a cycle-matroid invariant, and there is no obvious reason to expect it to be. In fact, a general Whitney twist has no evident action on rotation systems. This paper studies which invariants of $G$ determine $\Gamma_G$, and conversely what $\Gamma_G$ determines.

\subsection*{The cubic genus polynomial is a cycle-matroid invariant.}

Our first result shows that the cycle matroid determines $\Gamma_G$ for $2$-connected cubic graphs. 

\begin{theoremletter}\label{thmA}
Let $G$ and $H$ be finite simple $2$-connected cubic graphs. If $M(G)\cong M(H)$, then
\[
\Gamma_G(x)=\Gamma_H(x).
\]
Equivalently, the orientable genus polynomial of a finite simple $2$-connected cubic graph is invariant under balanced Whitney twists.
\end{theoremletter}

Whitney's $2$-isomorphism theorem~\cite{Whitney} states that two $2$-connected graphs have isomorphic cycle matroids exactly when one is obtained from the other by a sequence of Whitney twists. Theorem~\ref{thmA} says that the genus polynomial of a cubic graph is invariant under this equivalence. We note that the cubic hypothesis is \textit{necessary}. For non-cubic graphs, the genus polynomial is in general not a cycle-matroid invariant. Remark~\ref{rem:cubic-necessary} exhibits a $2$-isomorphic non-cubic pair with distinct genus polynomials. In principle, the matroid result could be weak for cubic graphs if the genus polynomial were already determined by standard invariants. Our second result shows this is not the case.

\begin{figure}[H]
\centering
\begin{subfigure}[t]{0.22\textwidth}\centering
\resizebox{\linewidth}{!}{%
\begin{tikzpicture}[line cap=round,
   v/.style={circle,fill=specblue!85!black,inner sep=1.5pt},
   e/.style={specblue!50!black,line width=0.8pt}]
  \draw[e] (1.112,-0.624)--(1.378,0.022);
  \draw[e] (1.112,-0.624)--(0.293,-1.29);
  \draw[e] (1.112,-0.624)--(1.012,0.468);
  \draw[e] (-1.261,0.451)--(-1.55,-0.493);
  \draw[e] (-1.261,0.451)--(-0.639,0.995);
  \draw[e] (-1.261,0.451)--(-0.311,0.341);
  \draw[e] (-0.764,-1.366)--(-1.55,-0.493);
  \draw[e] (-0.764,-1.366)--(-1.134,-0.842);
  \draw[e] (-0.764,-1.366)--(0.293,-1.29);
  \draw[e] (-0.6,-0.029)--(-0.639,0.995);
  \draw[e] (-0.6,-0.029)--(-1.134,-0.842);
  \draw[e] (-0.6,-0.029)--(0.442,-0.225);
  \draw[e] (0.352,1.313)--(-0.639,0.995);
  \draw[e] (0.352,1.313)--(1.012,0.468);
  \draw[e] (0.352,1.313)--(1.204,1.066);
  \draw[e] (-0.002,-0.67)--(-0.311,0.341);
  \draw[e] (-0.002,-0.67)--(0.293,-1.29);
  \draw[e] (-0.002,-0.67)--(0.442,-0.225);
  \draw[e] (0.467,0.882)--(-0.311,0.341);
  \draw[e] (0.467,0.882)--(1.012,0.468);
  \draw[e] (0.467,0.882)--(1.204,1.066);
  \draw[e] (1.378,0.022)--(0.442,-0.225);
  \draw[e] (1.378,0.022)--(1.204,1.066);
  \draw[e] (-1.55,-0.493)--(-1.134,-0.842);
  \node[v] at (1.112,-0.624) {};
  \node[v] at (-1.261,0.451) {};
  \node[v] at (-0.764,-1.366) {};
  \node[v] at (-0.6,-0.029) {};
  \node[v] at (0.352,1.313) {};
  \node[v] at (-0.002,-0.67) {};
  \node[v] at (0.467,0.882) {};
  \node[v] at (1.378,0.022) {};
  \node[v] at (-1.55,-0.493) {};
  \node[v] at (-0.639,0.995) {};
  \node[v] at (-0.311,0.341) {};
  \node[v] at (-1.134,-0.842) {};
  \node[v] at (0.293,-1.29) {};
  \node[v] at (0.442,-0.225) {};
  \node[v] at (1.012,0.468) {};
  \node[v] at (1.204,1.066) {};
\end{tikzpicture}}
\caption*{\scriptsize $\Gamma_{G_1}=12x+2036x^{2}+26432x^{3}+37056x^{4}$}
\end{subfigure}
\hfill
\begin{subfigure}[t]{0.22\textwidth}\centering
\resizebox{\linewidth}{!}{%
\begin{tikzpicture}[line cap=round,
   v/.style={circle,fill=specblue!85!black,inner sep=1.5pt},
   e/.style={specblue!50!black,line width=0.8pt}]
  \draw[e] (-0.494,-0.338)--(0.366,0.207);
  \draw[e] (-0.494,-0.338)--(-1.24,0.176);
  \draw[e] (-0.494,-0.338)--(-0.181,-1.187);
  \draw[e] (0.737,-1.021)--(1.186,-0.182);
  \draw[e] (0.737,-1.021)--(-0.181,-1.187);
  \draw[e] (0.737,-1.021)--(0.74,-0.405);
  \draw[e] (-1.55,-0.222)--(-1.225,0.794);
  \draw[e] (-1.55,-0.222)--(-0.91,-0.897);
  \draw[e] (-1.55,-0.222)--(-1.24,0.176);
  \draw[e] (-0.314,1.164)--(-1.225,0.794);
  \draw[e] (-0.314,1.164)--(-0.212,0.527);
  \draw[e] (-0.314,1.164)--(0.635,0.937);
  \draw[e] (-0.142,-0.402)--(-0.91,-0.897);
  \draw[e] (-0.142,-0.402)--(0.74,-0.405);
  \draw[e] (-0.142,-0.402)--(-0.212,0.527);
  \draw[e] (1.147,0.606)--(0.74,-0.405);
  \draw[e] (1.147,0.606)--(0.635,0.937);
  \draw[e] (1.147,0.606)--(1.458,0.244);
  \draw[e] (0.366,0.207)--(-0.212,0.527);
  \draw[e] (0.366,0.207)--(1.458,0.244);
  \draw[e] (1.186,-0.182)--(0.635,0.937);
  \draw[e] (1.186,-0.182)--(1.458,0.244);
  \draw[e] (-1.225,0.794)--(-1.24,0.176);
  \draw[e] (-0.91,-0.897)--(-0.181,-1.187);
  \node[v] at (-0.494,-0.338) {};
  \node[v] at (0.737,-1.021) {};
  \node[v] at (-1.55,-0.222) {};
  \node[v] at (-0.314,1.164) {};
  \node[v] at (-0.142,-0.402) {};
  \node[v] at (1.147,0.606) {};
  \node[v] at (0.366,0.207) {};
  \node[v] at (1.186,-0.182) {};
  \node[v] at (-1.225,0.794) {};
  \node[v] at (-0.91,-0.897) {};
  \node[v] at (-1.24,0.176) {};
  \node[v] at (-0.181,-1.187) {};
  \node[v] at (0.74,-0.405) {};
  \node[v] at (-0.212,0.527) {};
  \node[v] at (0.635,0.937) {};
  \node[v] at (1.458,0.244) {};
\end{tikzpicture}}
\caption*{\scriptsize $\Gamma_{G_2}=20x+2028x^{2}+26304x^{3}+37184x^{4}$}
\end{subfigure}
\hfill
\begin{subfigure}[t]{0.22\textwidth}\centering
\resizebox{\linewidth}{!}{%
\begin{tikzpicture}[line cap=round,
   v/.style={circle,fill=specblue!85!black,inner sep=1.5pt},
   e/.style={specblue!50!black,line width=0.8pt}]
  \draw[e] (-0.674,-1.55)--(-1.296,-0.753);
  \draw[e] (-0.674,-1.55)--(0.129,-1.127);
  \draw[e] (-0.674,-1.55)--(-0.654,-1.049);
  \draw[e] (0.154,-0.491)--(0.861,0.263);
  \draw[e] (0.154,-0.491)--(-0.654,-1.049);
  \draw[e] (0.154,-0.491)--(-0.501,-0.283);
  \draw[e] (-0.887,0.702)--(-0.152,1.265);
  \draw[e] (-0.887,0.702)--(-0.501,-0.283);
  \draw[e] (-0.887,0.702)--(-1.147,0.189);
  \draw[e] (-0.171,0.69)--(1.015,0.884);
  \draw[e] (-0.171,0.69)--(-1.147,0.189);
  \draw[e] (-0.171,0.69)--(0.345,0.403);
  \draw[e] (0.846,-0.523)--(0.129,-1.127);
  \draw[e] (0.846,-0.523)--(0.345,0.403);
  \draw[e] (0.846,-0.523)--(1.473,0.039);
  \draw[e] (-1.296,-0.753)--(-0.654,-1.049);
  \draw[e] (-1.296,-0.753)--(-1.147,0.189);
  \draw[e] (0.861,0.263)--(1.473,0.039);
  \draw[e] (0.861,0.263)--(0.659,1.342);
  \draw[e] (-0.152,1.265)--(0.345,0.403);
  \draw[e] (-0.152,1.265)--(0.659,1.342);
  \draw[e] (1.015,0.884)--(1.473,0.039);
  \draw[e] (1.015,0.884)--(0.659,1.342);
  \draw[e] (0.129,-1.127)--(-0.501,-0.283);
  \node[v] at (-0.674,-1.55) {};
  \node[v] at (0.154,-0.491) {};
  \node[v] at (-0.887,0.702) {};
  \node[v] at (-0.171,0.69) {};
  \node[v] at (0.846,-0.523) {};
  \node[v] at (-1.296,-0.753) {};
  \node[v] at (0.861,0.263) {};
  \node[v] at (-0.152,1.265) {};
  \node[v] at (1.015,0.884) {};
  \node[v] at (0.129,-1.127) {};
  \node[v] at (-0.654,-1.049) {};
  \node[v] at (-0.501,-0.283) {};
  \node[v] at (-1.147,0.189) {};
  \node[v] at (0.345,0.403) {};
  \node[v] at (1.473,0.039) {};
  \node[v] at (0.659,1.342) {};
\end{tikzpicture}}
\caption*{\scriptsize $\Gamma_{G_3}=12x+2148x^{2}+25872x^{3}+37504x^{4}$}
\end{subfigure}
\caption{The graphs $G_1,G_2,G_3$ of Theorem~\ref{thmB}, matching the rows of Table~\ref{tab:n16-smallest-witness} top to bottom.}
\label{fig:witness16-intro}
\end{figure}
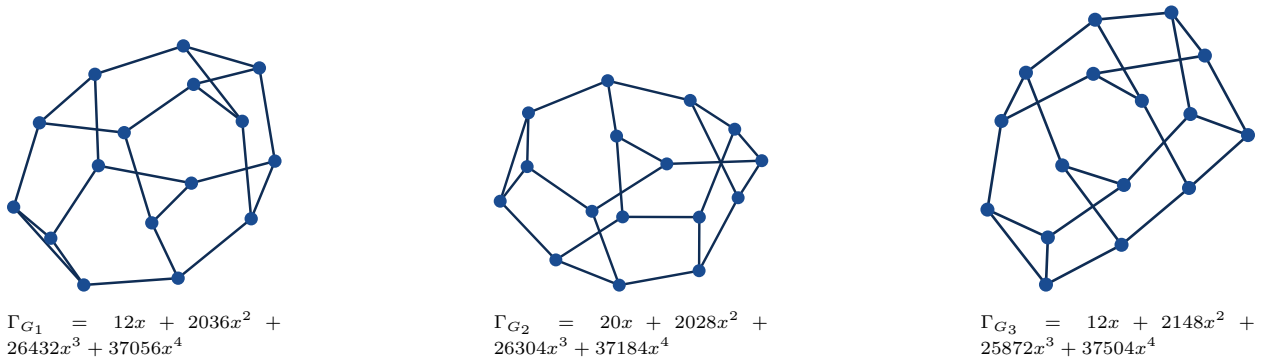

\begin{theoremletter}\label{thmB}
There exist finite simple connected cubic graphs, on as few as $16$ vertices, that have identical adjacency/Laplacian spectra, spanning tree counts, order, girth, diameter, vertex and edge connectivity, automorphism-group order, and the cycle counts through length $10$, yet have pairwise distinct orientable genus polynomials. In particular, the genus polynomial of a cubic graph is not determined by its spectrum.
\end{theoremletter}

The reverse implication is also false.

\begin{theoremletterprime}
There exist finite simple $3$-connected cubic graphs $G$ and $H$ on $12$ vertices with $\Gamma_G=\Gamma_H$ but distinct adjacency spectra. Together with Theorem~\ref{thmB}, this shows that the orientable genus polynomial of a cubic graph and its adjacency spectrum are \emph{incomparable}.
\end{theoremletterprime}

We give the pair in Example~\ref{ex:converse}.

Which graphs are determined by their spectrum is a long-standing question in spectral graph theory~\cite{vanDamHaemers,vanDamHaemers2009}, and cospectral pairs are common enough to have been enumerated systematically for all graphs on at most eleven vertices~\cite{HaemersSpence}. The standard way to show that an invariant is not spectral is to produce a cospectral pair on which it differs. Theorem~\ref{thmB} does this for the orientable genus polynomial, a role in which, to our knowledge, it has not previously appeared. Figure~\ref{fig:witness16-intro} shows 
the three graphs from Theorem~\ref{thmB}. A second family, six graphs on $22$ vertices, is introduced in Section~\ref{sec:witnesses}.

\subsection*{Short faces are spectral, long faces are not.}
Our third result explains the difference at the level of the expected face count by splitting faces by length. For a uniformly random rotation system, linearity of expectation gives the expected face count as a sum over facial candidates. We split it at twice the girth,
\begin{equation}\label{eq:intro-SL}
\Exp[F]=S(G)+L(G),
\qquad
S(G)=\sum_{k=\ell}^{2\ell-1}2^{1-k}c_k(G),
\end{equation}
where $S(G)$ is the expected contribution of faces of length less than $2\ell$ and $L(G)$ that of the longer faces. Faces of length less than $2\ell$ are forced to be ordinary simple cycles. Examining the expected face count is not new~\cite{StahlRegions} and continues in recent work~\cite{CampionLothMohar,ChenGao}. Theorem~\ref{thmC} computes the short part exactly.

\begin{theoremletter}\label{thmC}
Let $G$ be a connected cubic graph of girth $\ell$. The short-face contribution is determined by the non-backtracking trace of $G$,
\[
S(G)=\sum_{k=\ell}^{2\ell-1}\frac{\tr(B^k)}{k\,2^{k}},
\]
and is therefore an adjacency-spectral invariant of $G$. However, the long-face contribution $L(G)=\Exp[F]-S(G)$ is not. The families of Theorem~\ref{thmB} are cospectral with equal $S(G)$ but unequal $L(G)$.
\end{theoremletter}

Theorem~\ref{thmC} also allows us to examine what the expected face count contains outside of the spectrum, which is exactly the long-face term $L(G)$. To work with $L(G)$ further, we define \emph{locally compatible circuits} (Definition~\ref{def:compatibility}) as closed non-backtracking circuits whose \emph{successor constraints} can be simultaneously realized by some rotation system. This allows us to write
\begin{equation}\label{eq:intro-compatible}
\Exp[F]=\sum_{C}\chi(C)\,2^{-r(C)},
\end{equation}
where $\chi(C)$ records compatibility and $r(C)$ the number of constrained vertices. For the six graph family of Section~\ref{sec:witnesses}, a computation shows that the expected face-length profiles agree at every length up to $11$ and first begin to differ at length $12$. The simple-cycle contribution at length 12 is still common to all six graphs, so the first distinction is carried entirely by \emph{non-simple} facial walks (Theorem~\ref{thm:n22-first-compatible-difference}). 

\subsection*{An infinite family.}

We also exhibit an infinite family of adjacency-cospectral pairs with distinct genus polynomials.

\begin{theoremletter}\label{thmD}
There exists an infinite family of pairs $(G_t,H_t)$, $t\ge1$, of connected simple cubic graphs on $14+2t$ vertices such that, for every $t$, the graphs $G_t$ and $H_t$ are adjacency-cospectral, and their minimum genera differ:
\[
\gamma(G_t)=2,\qquad \gamma(H_t)=1.
\]
In particular $\Gamma_{G_t}(x)\ne\Gamma_{H_t}(x)$ for every $t$, the polynomials differing already in their least exponent.
\end{theoremletter}

We prove Theorem~\ref{thmD} in Section~\ref{sec:infinite-family}.

\subsection*{Further contributions.}
We complement the structural results with computations and applications of our short-face identity:
\begin{enumerate}[label=\textup{(\arabic*)}, leftmargin=2.2em]

\item a complete census of orientable genus polynomials for all connected cubic graphs through $22$ vertices, comprising $7{,}875{,}918$ graphs (Table~\ref{tab:census}), expanding the $n=16$ census of~\cite{CarrDhaliwalMohar}.

\item a deterministic lower bound on $\gamma(G)$, computed only by short-cycle counts and tested in Section~\ref{sec:bounds}.
\end{enumerate}

\subsection*{Related work.}
Carr and Mohar~\cite{CarrMohar} recently studied the genus distribution of a
\emph{fixed} cubic graph. They parametrized its $2^n$ embeddings and analyzed the genus through the \emph{unstable dual}, which they decompose across small vertex cuts. They show that the edges of a $2$-cut contribute only factors of two to a fixed graph's distribution, which they also relate to Whitney switching~\cite[\S5.2]{CarrMohar}.

\section{Preliminaries}
\label{sec:preliminaries}

All graphs are finite and undirected, and unless stated otherwise, simple. We follow the standard conventions of topological graph theory. For embeddings, rotation systems, and the genus, we refer the reader to Mohar--Thomassen~\cite{MoharThomassen}, and for matroid terminology to Oxley~\cite{Oxley}. Throughout, $G$ is a connected cubic graph unless specified otherwise, with
\[
n=|V(G)|,\qquad m=|E(G)|=\tfrac{3n}{2}.
\]
A graph is \emph{$2$-connected} if it is connected, has at least three vertices, and has no cut vertex, and \emph{$3$-connected} if it additionally has at least four vertices and no $2$-element vertex cut.

\subsection{Cycles and girth}

\begin{definition}
The \emph{girth} $\ell=\ell(G)$ is the length of a shortest cycle. For $k\ge 1$, $c_k(G)$ denotes the number of cycles of length $k$ in $G$, where a \emph{cycle} is an unoriented closed walk through $k$ distinct vertices, and cyclic shifts of the same closed walk are not distinguished as different cycles.
\end{definition}

\subsection{Darts, rotation systems, and faces}

\begin{definition}
A \emph{dart} of $G$ is an ordered pair $(u,v)$ with $uv\in E(G)$. We call $u$ its \emph{tail} and $v$ its \emph{head}, so each edge $uv$ gives two oppositely directed darts $(u,v)$ and $(v,u)$. Let $D(G)=\{(u,v):uv\in E(G)\}$ be the set of darts and $\alpha\colon D(G)\to D(G)$, $\alpha(u,v)=(v,u)$, the permutation exchanging the two darts of each edge.
\end{definition}

\begin{definition}
An \emph{orientable rotation system} $\rho$ assigns to each vertex $v$ a cyclic ordering $\sigma_v$ of the out-darts at $v$ (the darts with tail $v$). These assemble into the permutation $\sigma_\rho=\prod_v\sigma_v$ of $D(G)$, and the \emph{face permutation} of $\rho$ is
\[
  \phi_\rho=\sigma_\rho\alpha,\qquad \phi_\rho(u,v)=(v,\sigma_v(u)),
\]
where $\sigma_v(u)$ denotes the neighbor of $v$ following $u$ in the rotation at $v$. We write $F(\rho)$ for the number of cycles of $\phi_\rho$, and $\RS(G)$ for the set of rotation systems of $G$.
\end{definition}

To read a face off $\rho$, start at a dart $(u,v)$, reverse it with $\alpha$, and advance to the next out-dart in the rotation at the new tail,
\[
  (u,v)\ \xrightarrow{\ \alpha\ }\ (v,u)\ \xrightarrow{\ \sigma_v\ }\ (v,\sigma_v(u)).
\]
Because $\phi_\rho$ permutes the finite set $D(G)$, iterating from any dart returns to it and closes a walk in $G$. Such a closed walk is a facial boundary walk of the cellular embedding that $\rho$ determines, and the cycles of $\phi_\rho$ are exactly its faces. For a cubic graph $G$, each vertex has three out-darts and hence exactly two rotations, so $|\RS(G)|=2^n$. By the Heffter--Edmonds--Ringel theorem, $\RS(G)$ is in bijection with the $2$-cell orientable embeddings of $G$, up to orientation-preserving homeomorphism~\cite{Heffter,Edmonds,Ringel,GrossTucker,MoharThomassen}.

\begin{lemma}\label{lem:genus-from-faces}
For a connected cubic graph $G$ on $n$ vertices and any $\rho\in\RS(G)$,
\[
g(\rho)=1+\frac n4-\frac12\,F(\rho).
\]
\end{lemma}

This is well-known and follows from Euler's formula.

\subsection{The genus polynomial and the face-count polynomial}

\begin{definition}\label{def:genus-poly}
The \emph{orientable genus polynomial} of $G$ is $\Gamma_G(x)=\sum_g N_g(G)x^g$, where $N_g(G)=\#\{\rho\in\RS(G):g(\rho)=g\}$. Its least and greatest exponents are the \emph{minimum genus} $\gamma(G)=\min_{\rho\in\RS(G)}g(\rho)$ and the \emph{maximum genus} $\gmax(G)=\max_{\rho\in\RS(G)}g(\rho)$. The \emph{face-count polynomial} is
\[
\Phi_G(y)=\sum_{\rho\in\RS(G)}y^{F(\rho)}.
\]
\end{definition}

By Lemma~\ref{lem:genus-from-faces}, for a connected cubic graph on $n$ vertices the genus and face count of any embedding satisfy $g=1+\tfrac{n}{4}-\tfrac{F}{2}$, equivalently $F=2+\tfrac{n}{2}-2g$. Hence the coefficient of $x^{g}$ in $\Gamma_G$ equals the coefficient of $y^{\,2+n/2-2g}$ in $\Phi_G$, and the two polynomials are interchangeable. 

The \emph{length} of a face is the number of darts on its boundary walk, or equivalently, the length of the corresponding cycle of $\phi_\rho$. We write $F_{=k}(\rho)$ for the number of faces of $\rho$ of length $k$, and $F_{<t}(\rho)$, $F_{\ge t}(\rho)$ for the number of faces of $\rho$ of length $<t$ and $\ge t$, respectively, so $F(\rho)=\sum_{k}F_{=k}(\rho)$. Probabilities and expectations always refer to $\rho$ chosen uniformly at random from $\RS(G)$.

\subsection{The non-backtracking matrix}

\begin{definition}\label{def:nonbacktracking}
The \emph{non-backtracking} matrix $B$ of $G$~\cite{Hashimoto} is the $2m\times 2m$ matrix indexed by darts, with $B_{d,d'}=1$ if the head of $d$ is the tail of $d'$ and $d'\neq\alpha(d)$, and $B_{d,d'}=0$ otherwise. The $\tr(B^k)$ counts the number of closed non-backtracking walks of length $k$.
\end{definition}

For $d$-regular graphs, the eigenvalues of $B$ are determined by those of the adjacency matrix through the Bass--Ihara determinant~\cite{Ihara,Bass,KotaniSunada}, which we use in Section~\ref{sec:short-spectral} and give the precise statement there.

\subsection{The cycle matroid and Whitney twists}

\begin{definition}\label{def:cycle-matroid}
The \emph{cycle matroid} $M(G)$ has ground set $E(G)$ and circuits the edge sets of the cycles of $G$~\cite{Oxley}. A \emph{$2$-isomorphism} $G\to H$ is a bijection $E(G)\to E(H)$ carrying cycles to cycles, equivalently an isomorphism $M(G)\to M(H)$.
\end{definition}

We defer the definitions of Whitney separations, Whitney twists, and balanced twists to Section~\ref{sec:twists}.

\part*{Part I. A matroid upper bound}
\addcontentsline{toc}{part}{Part I. A matroid upper bound}

We begin with the upper bound: the matroid determines the polynomial. The cycle matroid records only which edge sets are cycles, nothing about how the edges meet at vertices or the graph's rotation systems. However, we show that for $2$-connected cubic graphs, that alone is enough to determine the entire genus polynomial.

\section{Whitney twists and balanced twists}
\label{sec:twists}

\begin{definition}\label{def:whitney-sep}
Let $G$ be $2$-connected. A \emph{Whitney separation} is an edge partition $E(G)=E(A)\du E(B)$ into edge-disjoint subgraphs with $V(A)\cap V(B)=\{u,v\}$, each of $A,B$ containing a $u$--$v$ path. Treating the cut vertices as separate copies $u_A,v_A$ and $u_B,v_B$, the graph $G$ is the \emph{straight gluing} $u_A{=}u_B,\,v_A{=}v_B$, and the \emph{Whitney twist} $\Gtau$ is the \emph{crossed gluing} $u_A{=}v_B,\,v_A{=}u_B$ (Figure~\ref{fig:whitney-twist}).
\end{definition}

\begin{proposition}[Whitney's $2$-isomorphism theorem]\label{prop:twist-matroid}
A Whitney twist preserves the cycle matroid: $M(\Gtau)\cong M(G)$ under the identity on edges. Conversely, any two $2$-isomorphic $2$-connected graphs are related by a sequence of Whitney twists~\cite{Whitney,Oxley}.
\end{proposition}

\begin{definition}\label{def:balanced}
With $a_u=\deg_A(u_A)$, $a_v=\deg_A(v_A)$, $b_u=\deg_B(u_B)$, $b_v=\deg_B(v_B)$, the twist is \emph{balanced} (degree-preserving) if $\Gtau$ has the same degrees at $u$ and $v$ as $G$. The twist changes $\deg(u)=a_u+b_u$ and $\deg(v)=a_v+b_v$ to $\deg(u)=a_u+b_v$ and $\deg(v)=a_v+b_u$. Thus, it is balanced if and only if $b_u=b_v$, which for a cubic graph is equivalent to $a_u=a_v$. In this case, we also call the separation and cut balanced.
\end{definition}

\begin{lemma}\label{lem:dichotomy}
If $G$ is cubic and $(A,B)$ is a Whitney separation, then $a_u+b_u=a_v+b_v=3$ and all four of these degrees are at least $1$. The twist is balanced if and only if, after possibly exchanging $A$ and $B$, the degree pairs are $(a_u,a_v)=(1,1)$ and $(b_u,b_v)=(2,2)$.
\end{lemma}

\begin{proof}
Each cut vertex has degree $3$ in $G$, and its incident edges are partitioned between the two sides, so $a_u+b_u=a_v+b_v=3$. Both sides contain a $u$--$v$ path, hence each of $a_u,a_v,b_u,b_v$ is at least $1$. Thus, each pair $\{a_u,b_u\}$ and $\{a_v,b_v\}$ is $\{1,2\}$ in some order. By Definition~\ref{def:balanced}, the twist is balanced precisely when $a_u=a_v$. If $a_u=a_v=1$ then $b_u=b_v=2$, the pattern $(1,1)\mid(2,2)$. If $a_u=a_v=2$, then $b_u=b_v=1$, which is the same pattern after exchanging the labels of $A$ and $B$. The only remaining case is $a_u\neq a_v$, giving $\{a_u,a_v\}=\{1,2\}$ and $\{b_u,b_v\}=\{1,2\}$, the unbalanced split $(1,2)\mid(2,1)$.
\end{proof}

\begin{lemma}[A cubic twist stays cubic iff balanced]\label{lem:cubic-iff-balanced}
Let $G$ be cubic and $\Gtau$ a Whitney twist of $G$. Then $\Gtau$ is cubic if and only if the twist is balanced. In particular, any Whitney twist sending one cubic graph to another is necessarily balanced, whereas the unbalanced split $(1,2)\mid(2,1)$ in Lemma~\ref{lem:dichotomy} would produce a vertex of degree $4$ and a vertex of degree $2$.
\end{lemma}

\begin{proof}
See Figure~\ref{fig:whitney-twist}.
\end{proof}

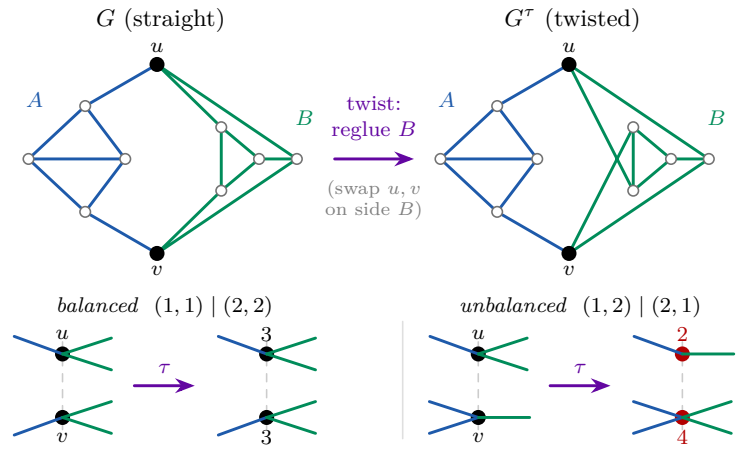
\begin{figure}[H]
\begin{center}
\begin{tikzpicture}[line cap=round, line join=round, >=Latex,
  ae/.style={aside,line width=1.1pt},
  be/.style={bside,line width=1.1pt},
  iv/.style={circle,draw=black!55,fill=white,inner sep=1.4pt,line width=0.6pt},
  cut/.style={circle,fill=black,inner sep=2.0pt},
  bad/.style={circle,fill=red!70!black,inner sep=2.0pt},
  clab/.style={font=\footnotesize},
  ttl/.style={font=\small},
  sep/.style={gray!40,dashed,line width=0.6pt},
  twist/.style={-{Stealth[length=2.6mm]},line width=1.1pt,royalpurple}]

\begin{scope}[shift={(0,0)}]
  \node[ttl] at (0.05,1.84) {$G$ (straight)};
  \coordinate (u)  at (0,1.25);     \coordinate (v)  at (0,-1.25);
  \coordinate (a1) at (-0.95,0.7);  \coordinate (a2) at (-1.7,0);
  \coordinate (a3) at (-0.95,-0.7); \coordinate (a4) at (-0.42,0);
  \coordinate (b1) at (1.85,0);     \coordinate (b2) at (0.85,0.42);
  \coordinate (b3) at (0.85,-0.42); \coordinate (b4) at (1.35,0);
  \draw[ae] (u)--(a1) (a1)--(a2) (a2)--(a3) (a3)--(v) (a4)--(a1) (a3)--(a4) (a2)--(a4);
  \draw[be] (b2)--(b3) (b4)--(b1) (b4)--(b2) (b4)--(b3);
  \draw[be] (u)--(b1) (u)--(b2) (v)--(b1) (v)--(b3);
  \node[clab,text=aside]  at (-1.62,0.78) {$A$};
  \node[clab,text=bside]  at (1.95,0.55)  {$B$};
  \foreach \p in {a1,a2,a3,a4,b1,b2,b3,b4} \node[iv] at (\p) {};
  \node[cut] at (u) {}; \node[cut] at (v) {};
  \node[clab,above=1pt] at (u) {$u$}; \node[clab,below=1pt] at (v) {$v$};
\end{scope}

\begin{scope}[shift={(5.45,0)}]
  \node[ttl] at (0.05,1.84) {$G^{\tau}$ (twisted)};
  \coordinate (u)  at (0,1.25);     \coordinate (v)  at (0,-1.25);
  \coordinate (a1) at (-0.95,0.7);  \coordinate (a2) at (-1.7,0);
  \coordinate (a3) at (-0.95,-0.7); \coordinate (a4) at (-0.42,0);
  \coordinate (b1) at (1.85,0);     \coordinate (b2) at (0.85,0.42);
  \coordinate (b3) at (0.85,-0.42); \coordinate (b4) at (1.35,0);
  \draw[ae] (u)--(a1) (a1)--(a2) (a2)--(a3) (a3)--(v) (a4)--(a1) (a3)--(a4) (a2)--(a4);
  \draw[be] (b2)--(b3) (b4)--(b1) (b4)--(b2) (b4)--(b3);
  \draw[be] (u)--(b1) (u)--(b3) (v)--(b1) (v)--(b2);
  \node[clab,text=aside]  at (-1.62,0.78) {$A$};
  \node[clab,text=bside]  at (1.95,0.55)  {$B$};
  \foreach \p in {a1,a2,a3,a4,b1,b2,b3,b4} \node[iv] at (\p) {};
  \node[cut] at (u) {}; \node[cut] at (v) {};
  \node[clab,above=1pt] at (u) {$u$}; \node[clab,below=1pt] at (v) {$v$};
\end{scope}

\draw[twist] (2.35,0)--(3.4,0);
\node[font=\footnotesize,align=center,anchor=south,text=royalpurple] at (2.87,0.12) {twist:\\ reglue $B$};
\node[font=\scriptsize,align=center,anchor=north,text=gray] at (2.87,-0.15) {(swap $u,v$\\ on side $B$)};

\begin{scope}[shift={(-1.25,-3.0)}]
  \node[font=\footnotesize] at (1.35,1.05){\emph{balanced}\ \ $(1,1)\mid(2,2)$};
  \node[font=\footnotesize] at (6.85,1.05){\emph{unbalanced}\ \ $(1,2)\mid(2,1)$};

  \draw[sep] (0,0.58)--(0,-0.58);
  \node[cut] at (0,0.42){};  \node[clab] at (0,0.68) {$u$};
  \draw[ae] (0,0.42)--++(160:0.68);
  \draw[be] (0,0.42)--++(18:0.68);  \draw[be] (0,0.42)--++(-18:0.68);
  \node[cut] at (0,-0.42){}; \node[clab] at (0,-0.68) {$v$};
  \draw[ae] (0,-0.42)--++(200:0.68);
  \draw[be] (0,-0.42)--++(18:0.68); \draw[be] (0,-0.42)--++(-18:0.68);
  \draw[twist] (0.95,0)--(1.75,0);
  \node[font=\footnotesize,text=royalpurple,anchor=south] at (1.35,0.04){$\tau$};

  \draw[sep] (2.7,0.58)--(2.7,-0.58);
  \node[cut] at (2.7,0.42){};  \node[clab] at (2.7,0.68) {$3$};
  \draw[ae] (2.7,0.42)--++(160:0.68);
  \draw[be] (2.7,0.42)--++(18:0.68);  \draw[be] (2.7,0.42)--++(-18:0.68);
  \node[cut] at (2.7,-0.42){}; \node[clab] at (2.7,-0.68) {$3$};
  \draw[ae] (2.7,-0.42)--++(200:0.68);
  \draw[be] (2.7,-0.42)--++(18:0.68); \draw[be] (2.7,-0.42)--++(-18:0.68);

  \draw[gray!25,line width=0.6pt] (4.5,0.85)--(4.5,-0.7);

  \draw[sep] (5.5,0.58)--(5.5,-0.58);
  \node[cut] at (5.5,0.42){};  \node[clab] at (5.5,0.68) {$u$};
  \draw[ae] (5.5,0.42)--++(160:0.68);
  \draw[be] (5.5,0.42)--++(18:0.68);  \draw[be] (5.5,0.42)--++(-18:0.68);
  \node[cut] at (5.5,-0.42){}; \node[clab] at (5.5,-0.68) {$v$};
  \draw[ae] (5.5,-0.42)--++(162:0.68); \draw[ae] (5.5,-0.42)--++(198:0.68);
  \draw[be] (5.5,-0.42)--++(0:0.68);
  \draw[twist] (6.45,0)--(7.25,0);
  \node[font=\footnotesize,text=royalpurple,anchor=south] at (6.85,0.04){$\tau$};

  \draw[sep] (8.2,0.58)--(8.2,-0.58);
  \node[bad] at (8.2,0.42){};  \node[clab,text=red!70!black] at (8.2,0.68) {$2$};
  \draw[ae] (8.2,0.42)--++(160:0.68);
  \draw[be] (8.2,0.42)--++(0:0.68);
  \node[bad] at (8.2,-0.42){}; \node[clab,text=red!70!black] at (8.2,-0.68) {$4$};
  \draw[ae] (8.2,-0.42)--++(162:0.68); \draw[ae] (8.2,-0.42)--++(198:0.68);
  \draw[be] (8.2,-0.42)--++(-18:0.68); \draw[be] (8.2,-0.42)--++(18:0.68);
\end{scope}
\end{tikzpicture}
\end{center}
\caption{A Whitney twist (top) and a visual proof of Lemma~\ref{lem:cubic-iff-balanced} (bottom).}
\label{fig:whitney-twist}
\end{figure}

A balanced split is symmetric in the two cut vertices, so a cubic graph maps to a cubic graph, while any other split does not (Lemma~\ref{lem:cubic-iff-balanced}). This is why, in Section~\ref{sec:cubic-whitney}, we may restrict to balanced twists.

\section{Invariance under a single balanced twist}
\label{sec:single-twist}

Fix a $2$-connected cubic graph $G=A\cup B$ with balanced cut $\{u,v\}$. By Lemma~\ref{lem:dichotomy} we may label the sides so that each cut vertex has one $A$-edge and two $B$-edges. Label the darts of $G$ as either \textit{$A$-darts} or \textit{$B$-darts} based on their corresponding edge. Every vertex other than $u$ and $v$ has either all $A$-darts or all $B$-darts, so any face permutation $\phi$ can only map an $A$-dart to a $B$-dart at $u$ or $v$. Denote the neighbors of $u$ as $a_u,b_1,b_2$ and the neighbors of $v$ as $a_v,c_1,c_2$, where $a_u,a_v$ are the $A$-neighbors. Fix $\rho\in\RS(G)$. At each cut vertex $x$, exactly one transition from an $A$-dart to a $B$-dart occurs: a facial walk leaves the $A$-side at $x$ only by arriving along the $A$-dart $(a_x,x)$ and continuing on the $B$-dart $(x,\sigma_x(a_x))$. Likewise, exactly one transition from a $B$-dart to an $A$-dart occurs at $x$: a facial walk enters the $A$-side at $x$ only on the dart $(x,a_x)$, whose $\phi$-predecessor is the $B$-dart $(b,x)$ with $\sigma_x(b)=a_x$. Which of the two $B$-neighbors plays the role of $b$, and which $B$-dart $\sigma_x(a_x)$ is, depends on $\rho$.

Consider the facial walk $\mathcal{F}$ containing the dart $(u,a_u)$. Its $\phi$-predecessor is a $B$-dart, so as $\mathcal{F}$ closes there must be a transition back to a $B$-dart. The first such transition after $(u,a_u)$ happens at a cut vertex, which we call $\pi_A(u)$. Define $\pi_A(v)$ in the same way from the dart $(v,a_v)$. If $\pi_A(u)=\pi_A(v)=w$, both walks reach the unique $A$-to-$B$ transition at $w$ through strings of $A$-darts. Since $\phi$ is injective, tracing backwards from that transition yields a single string of $A$-darts, which begins at the unique dart whose predecessor is a $B$-dart and must be one of $(u,a_u)$ or $(v,a_v)$. Hence $\pi_A$ is a permutation of $\{u,v\}$. Define $\pi_B$ symmetrically: $\pi_B(u)$ is the cut vertex carrying the first $B$-to-$A$ transition on the facial walk after the $B$-dart $(u,\sigma_u(a_u))$ (that walk contains the $A$-dart $(a_u,u)$, so such a transition exists), and likewise $\pi_B(v)$. The same argument makes $\pi_B$ a permutation. Both permutations depend on $\rho$, and we write $\pi_A^\rho,\pi_B^\rho$. There are four cases:
\[
\begin{array}{c|c|c|c}
\text{case} & \pi_A & \pi_B & \pi_B\pi_A \\
\hline
1 & \operatorname{id} & \operatorname{id} & \operatorname{id} \\
2 & (uv) & (uv) & \operatorname{id} \\
3 & \operatorname{id} & (uv) & (uv) \\
4 & (uv) & \operatorname{id} & (uv)
\end{array}
\]
We call a face containing both $A$- and $B$-darts a \textit{crossing} face. Such a face alternates between maximal strings of $A$-darts and maximal strings of $B$-darts. Starting at $x\in\{u,v\}$, one first follows a string of $A$-darts to $\pi_A(x)$, and then a string of $B$-darts to $\pi_B(\pi_A(x))$. Hence the crossing faces are precisely the orbits of the permutation $\pi_B\pi_A$ on $\{u,v\}$. Thus, in cases $1$ and $2$, where $\pi_B\pi_A=\operatorname{id}$, there are two crossing faces. One corresponds to the orbit $\{u\}$ and the other to the orbit $\{v\}$. In cases $3$ and $4$, where $\pi_B\pi_A=(uv)$, there is one crossing face, corresponding to the single orbit $\{u,v\}$. Therefore, the number of crossing faces is
$\cyc(\pi_B\pi_A)$. Now let $c_A(\rho)$ be the number of faces of $\rho$ using only $A$-darts, and let $c_B(\rho)$ be the number of faces of $\rho$ using only $B$-darts. Every face is either closed in $A$, closed in $B$, or crossing. Thus,
\begin{equation}\label{eq:face-decomposition}
F_G(\rho)
=
c_A(\rho)+c_B(\rho)+\cyc(\pi_B\pi_A).
\end{equation}

We now compute this expression before and after a balanced Whitney twist.

\begin{theorem}\label{thm:single-twist}
Let $G$ be a $2$-connected cubic graph and $\Gtau$ a balanced Whitney twist of $G$. There is an involution $\Theta:\RS(G)\to\RS(\Gtau)$, fixing all rotations away from the cut and interchanging the two $B$-blocks at the cut, such that
\[
F_{\Gtau}(\Theta(\rho))=F_G(\rho)\qquad\text{for every }\rho\in\RS(G).
\]
Consequently $\Phi_G(y)=\Phi_{\Gtau}(y)$ and $\Gamma_G(x)=\Gamma_{\Gtau}(x)$.
\end{theorem}
\begin{proof}

 Let $\tau=(uv)$ be the transposition of the two cut vertices. Define $\Theta\rho$ on $G^\tau$ as follows. Keep all rotations on the $A$-side unchanged. On the $B$-side, change the rotations by relabeling $u\leftrightarrow v$. Finally, at the cut vertices, keep the unique $A$-dart at the same cut vertex and exchange the two $B$-blocks (Figure~\ref{fig:theta-involution}). Applying the same process after twisting back returns $\rho$, so $\Theta$ is an involution from $\mathcal R(G)$ to $\mathcal R(G^\tau)$.

By construction
\[
c_A(\Theta\rho)=c_A(\rho),
\qquad
\pi_A^{\Theta\rho}=\pi_A^\rho.
\]
Sequences of $B$-darts are changed only by the swap $\tau=(uv)$: the rotation $\Theta\rho$ at $v$ restricted to $B$-darts is the $\tau$-relabeling of $\rho$'s rotation at $u$, and vice versa, so $\phi_{\Theta\rho}\circ\tau=\tau\circ\phi_\rho$ on $B$-darts. Thus closed $B$-faces correspond bijectively, and a sequence of $B$-darts from $x$ to $y$ in $G$ becomes a sequence of $B$-darts from $\tau(x)$ to $\tau(y)$ in $G^\tau$. Therefore
\[
c_B(\Theta\rho)=c_B(\rho),
\qquad
\pi_B^{\Theta\rho}
=
\tau \pi_B^\rho \tau^{-1}.
\]
Since $\Sym\{u,v\}\cong S_2$ is abelian, this gives
\[
\pi_B^{\Theta\rho}
=
\pi_B^\rho.
\]
Thus the crossing term is also preserved:
\[
\cyc
\bigl(\pi_B^{\Theta\rho}\pi_A^{\Theta\rho}\bigr)
=
\cyc
\bigl(\pi_B^\rho\pi_A^\rho\bigr).
\]
Substituting into \eqref{eq:face-decomposition}, we obtain
\[
F_{G^\tau}(\Theta\rho)=F_G(\rho).
\]
Therefore $\Theta$ is a face-count-preserving bijection
\[
\Theta:\mathcal R(G)\longrightarrow \mathcal R(G^\tau).
\]
It follows that
\[
\Phi_G(y)=\Phi_{G^\tau}(y).
\]
Finally, $G$ and $G^\tau$ have the same numbers of vertices and edges, so Euler's formula implies that paired rotation systems have the same genus. Hence
\[
\Gamma_G(x)=\Gamma_{G^\tau}(x).
\]

\end{proof}

\begin{figure}[htbp]
\centering
\begin{tikzpicture}[line cap=round, line join=round, >=Latex,
  cut/.style={circle,fill=black,inner sep=1.9pt},
  ablob/.style={draw=aside!60,fill=aside!7,line width=0.8pt},
  bblob/.style={draw=bside!60,fill=bside!7,line width=0.8pt},
  ae/.style={aside,line width=1.05pt},
  be/.style={bside,line width=1.05pt},
  astr/.style={aside,dashed,line width=1.0pt,-{Stealth[length=2.4mm]},shorten >=3pt,shorten <=3pt},
  bstr/.style={bside,dashed,line width=1.0pt,-{Stealth[length=2.4mm]},shorten >=3pt,shorten <=3pt},
  loop/.style={dotted,line width=0.9pt},
  clab/.style={font=\footnotesize},
  tlab/.style={font=\scriptsize},
  theta/.style={-{Stealth[length=2.6mm]},line width=1.1pt,black!70}]

\begin{scope}[shift={(0,0)}]
  \coordinate (u) at (0,1.15); \coordinate (v) at (0,-1.15);
  \draw[ablob] (u) .. controls (-1.95,1.2) and (-1.95,-1.2) .. (v)
                   .. controls (-0.55,-0.5) and (-0.55,0.5) .. (u);
  \draw[bblob] (u) .. controls (1.95,1.2) and (1.95,-1.2) .. (v)
                   .. controls (0.55,-0.5) and (0.55,0.5) .. (u);
  \node[clab,text=aside] at (-0.98,-0.62) {$A$};
  \node[clab,text=bside] at (0.98,-0.62) {$B$};
  \draw[loop,aside] (-0.98,0.42) circle[radius=0.22];
  \node[tlab,text=aside] at (-0.98,0.42) {$c_A$};
  \draw[loop,bside] (0.98,0.42) circle[radius=0.22];
  \node[tlab,text=bside] at (0.98,0.42) {$c_B$};
  \draw[astr] (u) .. controls (-0.85,0.5) and (-0.85,-0.5) .. (v);
  \draw[bstr] (v) .. controls (0.85,-0.5) and (0.85,0.5) .. (u);
  \node[tlab,text=aside,anchor=east] at (-0.72,0) {$\pi_A$};
  \node[tlab,text=bside,anchor=west] at (0.72,0) {$\pi_B$};
  \node[cut] at (u) {}; \node[cut] at (v) {};
  \node[clab,above=2pt] at (u) {$u$}; \node[clab,below=2pt] at (v) {$v$};
  \node[font=\scriptsize, anchor=north, align=center, text width=4.0cm] at (0,-1.62)
    {};
\end{scope}

\begin{scope}[shift={(4.05,0)}]
  \foreach \x/\hdr/\ubA/\ubB/\vbA/\vbB in {0/{$\rho$}/{b_1}/{b_2}/{c_1}/{c_2},
                                           3.35/{$\Theta\rho$}/{c_1}/{c_2}/{b_1}/{b_2}}{
  \begin{scope}[shift={(\x,0)}]
    \coordinate (u) at (0,0.72); \coordinate (v) at (0,-0.72);
    \fill[bside!14] (u) -- ++(32:0.8) arc[start angle=32,end angle=-32,radius=0.8] -- cycle;
    \fill[bside!14] (v) -- ++(32:0.8) arc[start angle=32,end angle=-32,radius=0.8] -- cycle;
    \draw[ae] (u) -- ++(160:0.6);
    \draw[be] (u) -- ++(22:0.6); \draw[be] (u) -- ++(-22:0.6);
    \draw[ae] (v) -- ++(200:0.6);
    \draw[be] (v) -- ++(22:0.6); \draw[be] (v) -- ++(-22:0.6);
    \node[tlab,anchor=east] at ($(u)+(160:0.66)$) {$a_u$};
    \node[tlab,anchor=west] at ($(u)+(22:0.66)$)  {$\ubA$};
    \node[tlab,anchor=west] at ($(u)+(-22:0.66)$) {$\ubB$};
    \node[tlab,anchor=east] at ($(v)+(200:0.66)$) {$a_v$};
    \node[tlab,anchor=west] at ($(v)+(22:0.66)$)  {$\vbA$};
    \node[tlab,anchor=west] at ($(v)+(-22:0.66)$) {$\vbB$};
    \node[cut] at (u) {}; \node[cut] at (v) {};
    \node[clab,above=3pt] at (u) {$u$}; \node[clab,below=3pt] at (v) {$v$};
    \node[font=\small] at (0,1.42) {\hdr};
  \end{scope}}
  \draw[theta] (1.28,0)--(2.08,0);
  \node[font=\footnotesize,anchor=south] at (1.68,0.08) {$\Theta$};
  \node[tlab] at (0,-1.42)    {$\sigma_u=(a_u\,\textcolor{bside!60!black}{b_1\,b_2})$};
  \node[tlab] at (0,-1.78)    {$\sigma_v=(a_v\,\textcolor{bside!60!black}{c_1\,c_2})$};
  \node[tlab] at (3.35,-1.42) {$\sigma_u=(a_u\,\textcolor{bside!60!black}{c_1\,c_2})$};
  \node[tlab] at (3.35,-1.78) {$\sigma_v=(a_v\,\textcolor{bside!60!black}{b_1\,b_2})$};
  \node[font=\scriptsize, anchor=north, align=center, text width=4.6cm] at (1.68,-2.12)
    {the $B$-blocks exchange poles};
\end{scope}

\begin{scope}[shift={(-1.4,-5.0)}]
  \foreach \x/\hdr/\rev in {0/{$G$}/0, 2.8/{$G^{\tau}$}/1}{
  \begin{scope}[shift={(\x,0)}]
    \coordinate (u) at (0,0.85); \coordinate (v) at (0,-0.85);
    \draw[bblob] (u) .. controls (1.75,0.95) and (1.75,-0.95) .. (v)
                     .. controls (0.3,-0.35) and (0.3,0.35) .. (u);
    \draw[loop,bside] (1.0,-0.32) circle[radius=0.18];
    \node[font=\tiny,text=bside] at (1.0,-0.32) {$c_B$};
    \ifnum\rev=0
      \draw[bstr] (v) .. controls (0.8,-0.48) and (0.8,0.48) .. (u);
    \else
      \draw[bstr] (u) .. controls (0.8,0.48) and (0.8,-0.48) .. (v);
    \fi
    \node[cut] at (u) {}; \node[cut] at (v) {};
    \node[tlab,anchor=east] at ($(u)+(-0.1,0)$) {$u$};
    \node[tlab,anchor=east] at ($(v)+(-0.1,0)$) {$v$};
    \node[font=\small] at (0.6,1.35) {\hdr};
  \end{scope}}
  \draw[theta] (1.45,0)--(2.2,0);
  \node[font=\footnotesize,anchor=south] at (1.82,0.08) {$\Theta$};
  \node[font=\scriptsize, anchor=north, align=center, text width=4.3cm] at (1.55,-1.62)
    {closed $B$-faces persist};
\end{scope}

\begin{scope}[shift={(4.35,-5.0)}]
  \foreach \x/\hdr in {0/{$G$}, 3.3/{$G^{\tau}$}}{
  \begin{scope}[shift={(\x,0)}]
    \coordinate (u) at (0,0.9); \coordinate (v) at (0,-0.9);
    \draw[ablob] (u) .. controls (-1.55,0.95) and (-1.55,-0.95) .. (v)
                     .. controls (-0.44,-0.4) and (-0.44,0.4) .. (u);
    \draw[bblob] (u) .. controls (1.55,0.95) and (1.55,-0.95) .. (v)
                     .. controls (0.44,-0.4) and (0.44,0.4) .. (u);
    \draw[astr] (u) .. controls (-0.72,0.42) and (-0.72,-0.42) .. (v);
    \draw[bstr] (v) .. controls (0.72,-0.42) and (0.72,0.42) .. (u);
    \node[font=\tiny,text=aside,anchor=east] at (-0.66,0) {$\pi_A$};
    \node[font=\tiny,text=bside,anchor=west] at (0.66,0) {$\pi_B$};
    \node[cut] at (u) {}; \node[cut] at (v) {};
    \node[tlab,above=2pt] at (u) {$u$}; \node[tlab,below=2pt] at (v) {$v$};
    \node[font=\small] at (0,1.58) {\hdr};
  \end{scope}}
  \draw[theta] (1.45,0)--(2.25,0);
  \node[font=\footnotesize,anchor=south] at (1.85,0.08) {$\Theta$};
  \node[font=\scriptsize, anchor=north, align=center, text width=4.9cm] at (1.65,-1.62)
    {$\pi_A$ and $\pi_B$ are unchanged};
\end{scope}
\end{tikzpicture}
\caption{The involution $\Theta$ of Theorem~\ref{thm:single-twist}.}
\label{fig:theta-involution}
\end{figure}
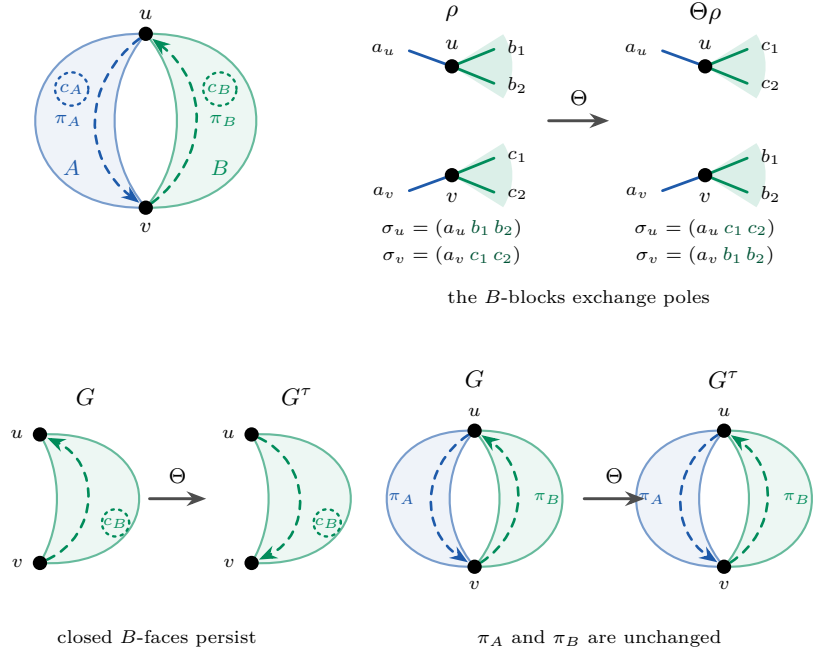

Thus, a balanced twist changes the rotation systems but not the face-count polynomial.

\section{Cubic representations of the cycle matroid}
\label{sec:cubic-whitney}

This section uses Tutte's decomposition of a $2$-connected graph into its $3$-connected components~\cite{Tutte}, in the form of an SPQR tree~\cite{DiBattistaTamassia}: a tree whose nodes are $3$-connected graphs ($R$-nodes), cycles ($S$-nodes), and bonds ($P$-nodes), in which adjacent nodes are joined by identifying a \emph{virtual edge} of each, and no two $S$-nodes and no two $P$-nodes are adjacent. Each virtual edge denotes a $2$-separation, which we say its tree edge \emph{displays}. The vertices of a node are called its \emph{skeleton vertices} and its edges its \emph{members}: a real member is an edge of the graph, and a virtual member stands in for the whole side beyond it. Figure~\ref{fig:spqr-example} shows a small cubic example. Starting from a specified SPQR tree, we can assemble a graph by making two kinds of choices: the identification at each tree edge and the cyclic arrangement of the members of each $S$-node. Different choices can produce non-isomorphic graphs. 

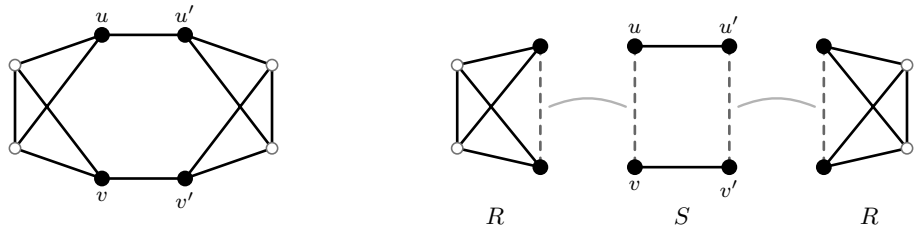
\begin{figure}[H]
\begin{center}
\begin{tikzpicture}[line cap=round, line join=round, >=Latex,
  iv/.style={circle,draw=black!55,fill=white,inner sep=1.5pt,line width=0.6pt},
  cut/.style={circle,fill=black,inner sep=2.1pt},
  ge/.style={line width=1.0pt},
  virt/.style={dashed,black!60,line width=1.0pt},
  ident/.style={black!30,line width=0.9pt},
  clab/.style={font=\footnotesize},
  ttl/.style={font=\small}]
\begin{scope}[shift={(0,0)}]
  \coordinate (A1) at (-1.7,0.55); \coordinate (B1) at (-1.7,-0.55);
  \coordinate (U1) at (-0.55,0.95); \coordinate (V1) at (-0.55,-0.95);
  \coordinate (U2) at (0.55,0.95);  \coordinate (V2) at (0.55,-0.95);
  \coordinate (A2) at (1.7,0.55);   \coordinate (B2) at (1.7,-0.55);
  \draw[ge] (U1)--(A1) (U1)--(B1) (A1)--(B1) (A1)--(V1) (B1)--(V1);
  \draw[ge] (U2)--(A2) (U2)--(B2) (A2)--(B2) (A2)--(V2) (B2)--(V2);
  \draw[ge] (U1)--(U2) (V1)--(V2);
  \node[iv] at (A1) {}; \node[iv] at (B1) {}; \node[iv] at (A2) {}; \node[iv] at (B2) {};
  \node[cut] at (U1) {}; \node[cut] at (V1) {}; \node[cut] at (U2) {}; \node[cut] at (V2) {};
  \node[clab,above=1.5pt] at (U1) {$u$};  \node[clab,below=1.5pt] at (V1) {$v$};
  \node[clab,above=1.5pt] at (U2) {$u'$}; \node[clab,below=1.5pt] at (V2) {$v'$};
\end{scope}
\begin{scope}[shift={(4.15,0)}]
  \coordinate (Ra) at (0,0.55); \coordinate (Rb) at (0,-0.55);
  \coordinate (Ru) at (1.1,0.8); \coordinate (Rv) at (1.1,-0.8);
  \draw[ge] (Ru)--(Ra) (Ru)--(Rb) (Ra)--(Rb) (Ra)--(Rv) (Rb)--(Rv);
  \draw[virt] (Ru)--(Rv);
  \node[iv] at (Ra) {}; \node[iv] at (Rb) {};
  \node[cut] at (Ru) {}; \node[cut] at (Rv) {};
  \node[ttl] at (0.5,-1.45) {$R$};
  \coordinate (Su) at (2.35,0.8); \coordinate (Sup) at (3.6,0.8);
  \coordinate (Sv) at (2.35,-0.8); \coordinate (Svp) at (3.6,-0.8);
  \draw[ge] (Su)--(Sup) (Sv)--(Svp);
  \draw[virt] (Su)--(Sv) (Sup)--(Svp);
  \node[cut] at (Su) {}; \node[cut] at (Sup) {}; \node[cut] at (Sv) {}; \node[cut] at (Svp) {};
  \node[clab,above=1.5pt] at (Su) {$u$};  \node[clab,below=1.5pt] at (Sv) {$v$};
  \node[clab,above=1.5pt] at (Sup) {$u'$}; \node[clab,below=1.5pt] at (Svp) {$v'$};
  \node[ttl] at (2.975,-1.45) {$S$};
  \coordinate (Qu) at (4.85,0.8); \coordinate (Qv) at (4.85,-0.8);
  \coordinate (Qa) at (5.95,0.55); \coordinate (Qb) at (5.95,-0.55);
  \draw[ge] (Qu)--(Qa) (Qu)--(Qb) (Qa)--(Qb) (Qa)--(Qv) (Qb)--(Qv);
  \draw[virt] (Qu)--(Qv);
  \node[iv] at (Qa) {}; \node[iv] at (Qb) {};
  \node[cut] at (Qu) {}; \node[cut] at (Qv) {};
  \node[ttl] at (5.45,-1.45) {$R$};
  \draw[ident] (1.22,0) to[bend left=22] (2.23,0);
  \draw[ident] (3.72,0) to[bend left=22] (4.73,0);
\end{scope}
\end{tikzpicture}
\end{center}
\caption{A cubic graph (left) and its SPQR decomposition (right).}
\label{fig:spqr-example}
\end{figure}

\begin{lemma}[Whitney's $2$-isomorphism theorem, $2$-connected form~\cite{Whitney,Truemper}]\label{lem:whitney-2iso}
Let $G$ and $H$ be $2$-connected graphs. Then $M(G)\cong M(H)$ if and only if a graph isomorphic to $H$ can be obtained from $G$ by a finite sequence of Whitney twists, each performed across a $2$-separation.
\end{lemma}

\begin{proof}
The ``if'' direction is Proposition~\ref{prop:twist-matroid}. The ``only if'' direction is Truemper's: two $2$-connected graphs with isomorphic cycle matroids are joined by a sequence of at most $n-2$ Whitney twists, each across a $2$-separation~\cite[Theorem~1]{Truemper}. 
\end{proof}

Theorem~\ref{thm:cubic-whitney} strengthens this statement for cubic graphs, showing the twists can be chosen to be balanced so that every intermediate graph is also cubic. Part (i) of the next proposition is not new. The tree is due to Tutte and to Cunningham--Edmonds, and the $3$-connected pieces to Whitney. Part (ii) is elementary, and we only prove it because the proof of Theorem~\ref{thm:cubic-whitney} uses its exact form.

\begin{proposition}[Graphic realizations of a connected matroid]\label{prop:realizations}
Let $M$ be a connected graphic matroid and let $G,H$ be $2$-connected graphs with $M(G)\cong M(H)\cong M$.
\begin{enumerate}
\item[\textup{(i)}] $G$ and $H$ have isomorphic Tutte decompositions: there is an isomorphism of their decomposition trees sending each node of $G$ to a node of $H$ of the same type, isomorphic as a graph, and sending the same elements of $M$. The tree, the node graphs, and the assignment of elements to nodes are invariants of $M$. Moreover, for each tree edge, the partition of $E(M)$ into the elements of the two subtrees is, in every $2$-connected realization of $M$, the edge partition of a Whitney separation whose cut vertices are the ends of the two identified virtual edges.
\item[\textup{(ii)}] Fix such isomorphisms. A realization is specified by its assembly data: for each $S$-node, the cyclic arrangement of its members and, for each tree edge, one of the two identifications of the ends of its two virtual edges. A Whitney twist across the two skeleton vertices of a $P$-node partitions its members into two nonempty sets. Twisting one side flips the identification at each tree edge whose virtual member lies on that side and changes nothing else. Thus, $P$-nodes require no arrangement data. Realizations with the same assembly data are isomorphic. A Whitney twist across a tree-edge separation flips the identification at that tree edge and changes nothing else. A Whitney twist across a separation whose cut is a pair of skeleton vertices of one $S$-node reverses an arc of that node's arrangement, flips the identification at each tree edge whose member lies on the arc, and changes nothing else.
\end{enumerate}
\end{proposition}
 \begin{proof}
(i) For the canonical tree, see \cite{Tutte,CunninghamEdmonds} and \cite[Theorem~8.3.10]{Oxley}. For the graph side, see \cite{Tutte,HopcroftTarjan,DiBattistaTamassia}. Uniqueness gives the invariance, and each node graph is determined by its type. An $R$-node is a $3$-connected matroid and has a unique graph by \cite[Lemma~5.3.2]{Oxley}. An $S$-node is a circuit, realized only by a cycle. A $P$-node is a parallel class, realized only by a set of parallel edges. For the last statement, consider a tree edge of the SPQR tree of $G$. It is realized by the $2$-vertex cut at the common ends of its two identified virtual edges, separating the elements of the two subtrees, and both sides contain a path between the cut vertices because $G$ is $2$-connected. Thus, the partition is the edge partition of a Whitney separation of $G$, and likewise of $H$.

(ii) Reconstruction is shown by induction on the number of nodes. By (i), a leaf node's elements form one side of a Whitney separation of the realization, that side is determined by the node's data, and it glues onto the rest in the one way its identification prescribes, so equal data give isomorphic realizations. For the twist cases, consider the two sides of the cut. Across a tree-edge separation, the sides are the two subtrees, so the twist exchanges the two gluings at that tree edge and moves no edge within either side. Across a pair of skeleton vertices of an $S$-node, the sides are the two arcs of that cycle, so re-gluing one arc with its ends exchanged reverses it in the arrangement and, at each member on the arc, swaps the roles of the member's two ends, which is the identification flip at its tree edge. Everything not on the arc is unchanged. For a $P$-node, the two sides are sets of parallel members with the same two ends. Re-gluing one side exchanges those ends, so it flips exactly the identifications at the tree edges of the virtual members on that side. Real members are unchanged, and there is no arrangement to change.
\end{proof}

Figure~\ref{fig:tree-edge-datum} shows the identification data: the same two pieces, with the two gluings.
 
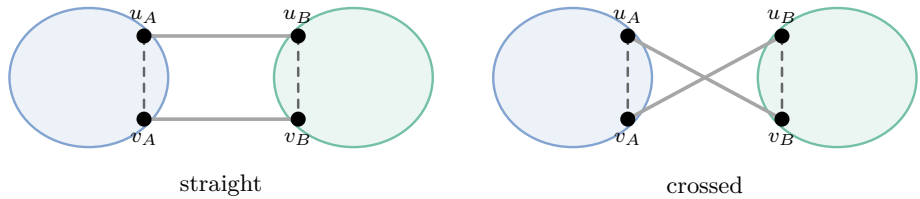
\begin{figure}[H]
\begin{center}
\begin{tikzpicture}[line cap=round, line join=round,
  cut/.style={circle,fill=black,inner sep=2.0pt},
  virt/.style={dashed,black!60,line width=0.95pt},
  glue/.style={black!35,line width=1.35pt},
  clab/.style={font=\footnotesize},
  ttl/.style={font=\small}]
\foreach \sx/\cross/\nm in {0/0/straight, 6.4/1/crossed}{
\begin{scope}[shift={(\sx,0)}]
  \draw[aside!55,fill=aside!8,line width=0.9pt] (-1.15,0) ellipse [x radius=1.05, y radius=0.92];
  \draw[bside!55,fill=bside!8,line width=0.9pt] (2.35,0) ellipse [x radius=1.05, y radius=0.92];
  \coordinate (uA) at (-0.42,0.55); \coordinate (vA) at (-0.42,-0.55);
  \coordinate (uB) at (1.62,0.55);  \coordinate (vB) at (1.62,-0.55);
  \draw[virt] (uA)--(vA); \draw[virt] (uB)--(vB);
  \ifnum\cross=0
    \draw[glue] (uA)--(uB); \draw[glue] (vA)--(vB);
  \else
    \draw[glue] (uA)--(vB); \draw[glue] (vA)--(uB);
  \fi
  \node[cut] at (uA) {}; \node[cut] at (vA) {}; \node[cut] at (uB) {}; \node[cut] at (vB) {};
  \node[clab,above=1.5pt] at (uA) {$u_A$}; \node[clab,below=1.5pt] at (vA) {$v_A$};
  \node[clab,above=1.5pt] at (uB) {$u_B$}; \node[clab,below=1.5pt] at (vB) {$v_B$};
  \node[ttl] at (0.6,-1.42) {\nm};
\end{scope}}
\end{tikzpicture}
\end{center}
\caption{The identification data at a tree edge.}
\label{fig:tree-edge-datum}
\end{figure}
 
Whether a separation admits a balanced twist is decided by the edge partition alone.
 
\begin{lemma}[Unbalanced separations are rigid]\label{lem:balanced-canonical}
Let $\{u,v\}$ be a $2$-separation of a $2$-connected cubic graph $G$, with sides $A,B$, so that $E(G)$ is partitioned as $E(A)\du E(B)$ and each side contains a $u$--$v$ path. Whether the separation is balanced depends only on this edge partition, not on the gluing. It is balanced if and only if side $A$ meets its two poles in equally many edges. If the separation is unbalanced, then exactly one of the two gluings of $A$ to $B$ is cubic, so every cubic graph with this cycle matroid and this $2$-separation uses that same gluing. If it is balanced, both gluings are cubic and differ by a balanced Whitney twist.
\end{lemma}
 
\begin{proof}
Regard each side as a graph with two poles. Its pole-degrees are determined by the side alone, hence by the edge partition. By Lemma~\ref{lem:dichotomy} each is 1 or 2, and the two degrees at each pole of $G$ sum to $3$, so the pole-degrees of $B$ complement those of $A$. By Definition~\ref{def:balanced}, the separation is balanced exactly when the pole-degrees of $A$ are equal. A gluing matches the poles of $A$ with the poles of $B$ in one of two ways, and the result is cubic if and only if the matched pole-degrees sum to $3$ at both poles. If the pole-degrees of $A$ are equal, say $(1,1)$ so that those of $B$ are $(2,2)$, then both matchings give $1+2=3$: both gluings are cubic, and they differ by the Whitney twist at $\{u,v\}$, which is balanced by Lemma~\ref{lem:cubic-iff-balanced}. If the pole-degrees of $A$ are $\{1,2\}$, only the matching that pairs $A$'s degree-$1$ pole with $B$'s degree-$2$ pole sums to $3$ at both poles, so exactly one gluing is cubic.
\end{proof}

\begin{lemma}[Balanced twists stay in the class]\label{lem:chain-maintenance}
Let $G$ be a finite simple $2$-connected cubic graph and let $G^\tau$ be obtained from $G$ by a balanced Whitney twist. Then $G^\tau$ is again finite, simple, $2$-connected, and cubic.
\end{lemma}
\begin{proof}
Cubicity is Lemma~\ref{lem:cubic-iff-balanced}. The twist preserves the cycle matroid $M$ (Proposition~\ref{prop:twist-matroid}), which is connected because $G$ is $2$-connected on at least three vertices. Hence, $G^\tau$ is loopless (a connected matroid on more than one element has no loops), and it is cubic on the same vertex set, so it has at least three vertices and no isolated vertices. Thus, by \cite[Proposition~4.1.7]{Oxley}, it is $2$-connected. Finally, since $G$ is simple, $M$ has no circuit of size at most $2$. A loop or a parallel pair of edges in $G^\tau$ would be such a circuit, so $G^\tau$ is simple.
\end{proof}

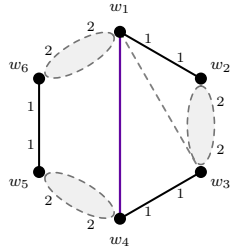
\begin{figure}[H]
\begin{center}
\resizebox{0.2\textwidth}{!}{
\begin{tikzpicture}[line cap=round, line join=round, >=Latex,
  cut/.style={circle,fill=black,inner sep=2.1pt},
  ge/.style={line width=1.0pt},
  lens/.style={dashed,black!55,fill=black!6,line width=0.8pt},
  num/.style={font=\scriptsize},
  clab/.style={font=\footnotesize}]
  \foreach \i/\ang in {1/90,2/30,3/-30,4/-90,5/210,6/150}{
    \coordinate (w\i) at (\ang:1.62);
  }
  \draw[ge] (w1)--(w2); \draw[ge] (w3)--(w4); \draw[ge] (w5)--(w6);
  \path (w6) -- (w1) coordinate[midway] (M1);
  \path (w2) -- (w3) coordinate[midway] (M3);
  \path (w4) -- (w5) coordinate[midway] (M5);
  \draw[lens,rotate around={30:(M1)}]  (M1) ellipse [x radius=0.68, y radius=0.235];
  \draw[lens,rotate around={90:(M3)}]  (M3) ellipse [x radius=0.68, y radius=0.235];
  \draw[lens,rotate around={150:(M5)}] (M5) ellipse [x radius=0.68, y radius=0.235];
  \draw[royalpurple,line width=1.1pt] (w1)--(w4);
  \draw[black!50,dashed,line width=0.9pt] (w1)--(w3);
  \foreach \i/\ang in {1/90,2/30,3/-30,4/-90,5/210,6/150}{
    \node[cut] at (w\i) {};
    \node[clab] at (\ang:2.02) {$w_{\i}$};
  }
  \foreach \a/\b/\va/\vb in {w1/w2/1/1, w3/w4/1/1, w5/w6/1/1}{
    \node[num] at ($($(\a)!0.3!(\b)$)!0.14cm!90:(\b)$) {$\va$};
    \node[num] at ($($(\b)!0.3!(\a)$)!-0.14cm!90:(\a)$) {$\vb$};
  }
  \foreach \a/\b in {w6/w1, w2/w3, w4/w5}{
    \node[num] at ($($(\a)!0.24!(\b)$)!0.34cm!90:(\b)$) {$2$};
    \node[num] at ($($(\b)!0.24!(\a)$)!-0.34cm!90:(\a)$) {$2$};
  }
\end{tikzpicture}
}
\end{center}
\caption{An $S$-node of a cubic graph. Real members (solid) and virtual members (dashed lenses) alternate. The twist across the solid chord $\{w_1,w_4\}$ has arcs of three members and is balanced. Across the dashed chord $\{w_1,w_3\}$, the arcs are even and it is unbalanced.}
\label{fig:snode-anatomy}
\end{figure}

\begin{lemma}[Cubic $S$-nodes]\label{lem:snode-structure}
Let $G$ be a simple $2$-connected cubic graph, and let $\nu$ be an $S$-node of its decomposition, each member \emph{contributing} to each of its two ends the number of $G$-edges it places there.
\begin{enumerate}
\item[\textup{(a)}] Real members contribute $1$ at each end, virtual members contribute $2$, and the two kinds alternate around the cycle.
\item[\textup{(b)}] A Whitney twist across two skeleton vertices of $\nu$ reverses one of the two arcs between them, and is balanced if and only if that arc has an odd number of members.
\item[\textup{(c)}] Every tree edge of the decomposition has an $S$- or $P$-node endpoint, and the Whitney separation it displays is balanced.
\end{enumerate}
\end{lemma}

\begin{proof}
(a) Each member contains a path between its two ends. A real edge is such a path, and each side of a Whitney separation contains one between its poles. Hence, every contribution is at least $1$, and since the contributions at a skeleton vertex sum to its degree $3$, each is $1$ or $2$. A real member contributes $1$ at each end. Suppose a virtual member $m$, standing for the $2$-pole graph $D$ with poles $w,w'$, contributed only $1$ at $w$, through the single edge $wq$ of $D$ at $w$. Then $q\neq w'$ (otherwise $D-wq$ would meet the rest of $G$ at $w'$ alone, a cut vertex), and every $w$--$w'$ path of $D$ passes through $q$, so $D$ is the series composition of $wq$ with a $2$-pole graph on $q,w'$ (Figure~\ref{fig:series-step}). In the reduced SPQR decomposition, such a nontrivial series composition would be represented by an $S$-node adjacent to the node containing $m$, unless it has already been merged into that node. Since the decomposition is reduced, adjacent $S$-nodes have been merged. Thus, this series structure must already lie on the cycle $\nu$ itself, meaning $q$ is a skeleton vertex of $\nu$ and $wq$ is a member of $\nu$, contradicting that $m$ is one member. Hence, virtual members contribute $2$ at each end, so each skeleton vertex meets one member of each kind. 

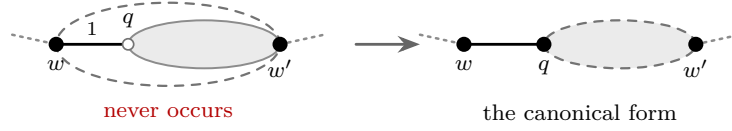
\begin{figure}[H]
\begin{center}
\resizebox{0.6\textwidth}{!}{
\begin{tikzpicture}[line cap=round, line join=round, >=Latex,
  iv/.style={circle,draw=black!55,fill=white,inner sep=1.4pt,line width=0.6pt},
  cut/.style={circle,fill=black,inner sep=1.9pt},
  ge/.style={line width=1.0pt},
  lens/.style={dashed,black!55,line width=0.8pt},
  blob/.style={fill=black!8,draw=black!45,line width=0.7pt},
  rest/.style={black!45,dotted,line width=0.9pt},
  num/.style={font=\scriptsize},
  arr/.style={-{Stealth[length=2.8mm]},line width=1.0pt,black!60},
  clab/.style={font=\footnotesize}]

\begin{scope}[shift={(0,0)}]
  \coordinate (w) at (0,0); \coordinate (q) at (0.9,0); \coordinate (wp) at (2.8,0);
  \draw[rest] (-0.55,0.12)--(w); \draw[rest] (wp)--(3.35,0.12);
  \draw[lens] (1.4,0) ellipse [x radius=1.4, y radius=0.52];
  \draw[blob] (1.85,0) ellipse [x radius=0.95, y radius=0.3];
  \draw[ge] (w)--(q);
  \node[iv] at (q) {};
  \node[cut] at (w) {}; \node[cut] at (wp) {};
  \node[num] at (0.45,0.2) {$1$};
  \node[clab,below=2pt] at (w) {$w$};
  \node[clab] at (0.9,0.3) {$q$};
  \node[clab,below=2pt] at (wp) {$w'$};
  \node[clab,text=red!70!black] at (1.4,-0.85) {never occurs};
\end{scope}
\draw[arr] (3.75,0)--(4.55,0);

\begin{scope}[shift={(5.1,0)}]
  \coordinate (w) at (0,0); \coordinate (q) at (1.0,0); \coordinate (wp) at (2.9,0);
  \draw[rest] (-0.55,0.12)--(w); \draw[rest] (wp)--(3.45,0.12);
  \draw[ge] (w)--(q);
  \draw[lens,fill=black!8] (1.95,0) ellipse [x radius=0.95, y radius=0.3];
  \node[cut] at (w) {}; \node[cut] at (q) {}; \node[cut] at (wp) {};
  \node[clab,below=2pt] at (w) {$w$}; \node[clab,below=2pt] at (q) {$q$};
  \node[clab,below=2pt] at (wp) {$w'$};
  \node[clab] at (1.45,-0.85) {the canonical form};
\end{scope}
\end{tikzpicture}
}
\end{center}
\caption{The side of a virtual member meeting $w$ through a single edge $wq$ (left) would be a series composition, which the canonical decomposition has already absorbed into the cycle (right).}
\label{fig:series-step}
\end{figure}

(b) The two sides of the twist are the two arcs between the cut vertices, and the twist reverses one of them (Proposition~\ref{prop:realizations}(ii)). The side degree of an arc at a cut vertex is the contribution of its end member there, $1$ if real and $2$ if virtual, so the twist is balanced exactly when the two end members of the arc have the same kind (Definition~\ref{def:balanced}). By part (a) the kinds alternate, so that happens exactly when the arc has an odd number of members (Figure~\ref{fig:snode-anatomy}).

(c) Suppose a tree edge joins two $R$-nodes, and let $u$ be one of its two cut vertices. Each $R$-node graph is $3$-connected, so on each side $u$ meets at least two members besides the virtual edge, and each contributes at least $1$ at $u$ as in (a). Then $\deg_G(u)\ge4$, contradicting cubicity. Hence, one endpoint is an $S$- or a $P$-node. If it is an $S$-node, the side beyond the tree edge is a virtual member of that node, with pole-degrees $(2,2)$ by (a), so the separation is balanced (Definition~\ref{def:balanced}). If it is a $P$-node with $p$ members, then $p\ge3$~\cite{DiBattistaTamassia}, and each member contributes at least $1$ at each pole. The poles have degree $3$, so $p=3$ and every member has pole-degrees $(1,1)$, and the side beyond the tree edge is one member. Thus, the separation is balanced.
\end{proof}

\begin{lemma}[Rearranging an $S$-node]\label{lem:snode-rearrange}
Let $G$ be a simple $2$-connected cubic graph, let $\nu$ be an $S$-node of its decomposition, and let $\alpha$ be assembly data for $\nu$ whose cyclic arrangement alternates between real and virtual members. Then finitely many balanced Whitney twists, each across a pair of skeleton vertices of $\nu$, with every intermediate graph simple, $2$-connected, and cubic, carry $G$ to a realization whose assembly data at $\nu$ is $\alpha$ and whose other assembly data are those of $G$.
\end{lemma}

\begin{proof}
We first record what a twist at $\nu$ can do. A pair of skeleton vertices of $\nu$ cuts the cycle into two arcs, and a Whitney twist across the pair reverses one of them (Lemma~\ref{lem:snode-structure}(b)). By Proposition~\ref{prop:realizations}(ii), it also flips the identification at the tree edge of every virtual member on that arc, and it changes nothing else: no other node's arrangement, no identification elsewhere. Hence, the twists available at $\nu$ are exactly arc reversals, and we must produce $\alpha$ by reversals alone. A reversal is balanced if and only if its arc has an odd number of members (Lemma~\ref{lem:snode-structure}(b)), and balanced twists keep every intermediate graph simple, $2$-connected, and cubic (Lemma~\ref{lem:chain-maintenance}), so every odd arc is available. The two shortest kinds suffice.

We now produce $\alpha$'s cyclic order, using three-member arcs. Reversing $m,m',m''$ exchanges $m$ and $m''$ and keeps $m'$ in its position. The ends of the arc sit two positions apart, so they have the same kind by the alternation in Lemma~\ref{lem:snode-structure}(a): each such reversal transposes two consecutive real members among the real positions, or two consecutive virtual members among the virtual positions, and moves nothing else in the cyclic order. Consecutive transpositions generate every rearrangement of the real members among their positions, one adjacent swap at a time, and likewise for the virtual members. The current arrangement and $\alpha$ are alternating cyclic orders of the same member set (Proposition~\ref{prop:realizations}(i)), so, overlaying them with kinds matching, they differ by exactly such a pair of rearrangements, and finitely many three-member reversals make the cyclic order $\alpha$'s.

These reversals also flipped identifications along their arcs, which is why we set the identifications last. For each virtual member whose identification now disagrees with $\alpha$'s, reverse the one-member arc consisting of that member alone, the twist across its two ends: the arc is odd, so the twist is balanced, and it fixes the cyclic order and flips exactly that identification (Proposition~\ref{prop:realizations}(ii)). Real members carry no identification. Finitely many such reversals set every identification at $\nu$ to the one $\alpha$ describes, and since every twist above changed only the arrangement of $\nu$ and identifications at tree edges of its members, the other assembly data are still those of $G$.
\end{proof}

\begin{remark}[Tree-edge twists alone do not suffice]\label{rem:arrangement-needed}
Let $A$ and $B$ be the two $2$-pole graphs of Figure~\ref{fig:bead-example}. Each has pole degrees $1$ at $p_1$ and $2$ at $p_2$. Stringing two copies of each in a cycle, identifying the pole $p_1$ of each copy with the pole $p_2$ of the next, puts $1+2=3$ edge ends at every pole, so the result is a simple $2$-connected cubic graph on $20$ vertices. The cyclic orders $A,A,B,B$ and $A,B,A,B$ give graphs $G_1$ and $G_2$ with the same cycle matroid on the same edge set. Its circuits are the cycles inside one copy together with the unions of one pole-to-pole path from each copy, and neither kind depends on the order. The common decomposition has a single $S$-node with eight members (Figure~\ref{fig:bead-arrangements}), and $G_1$, $G_2$ differ exactly in its arrangement. They are not isomorphic. Any isomorphism would carry the decomposition of $G_1$ to that of $G_2$ (Proposition~\ref{prop:realizations}(i)), hence the cyclic order of the piece types to itself up to rotation and reflection, and $A,A,B,B$ and $A,B,A,B$ differ even up to that. Tree-edge twists change no arrangement, so no sequence of them connects $G_1$ to $G_2$. But twisting across the pole shared by the two copies of $A$ and the vertex $d$ of the following copy of $B$, the arc between the cut vertices has three members, so the twist is balanced by Lemma~\ref{lem:snode-structure}(b), and reversing it exchanges its two virtual members and carries $G_1$ to a graph isomorphic to $G_2$ (the pair in Figure~\ref{fig:bead-arrangements}).
\end{remark}
\begin{figure}[H]
\begin{center}
\def\beadfigscale{0.8}
\scalebox{\beadfigscale}{%
\begin{tikzpicture}[line cap=round, line join=round, >=Latex,
  iv/.style={circle,draw=black!55,fill=white,inner sep=1.5pt,line width=0.6pt},
  cut/.style={circle,fill=black,inner sep=2.1pt},
  age/.style={aside,line width=1.05pt},
  bge/.style={bside,line width=1.05pt},
  abead/.style={aside,line width=3.2pt},
  bbead/.style={bside,line width=3.2pt},
  clab/.style={font=\footnotesize},
  ttl/.style={font=\small}]

\begin{scope}[shift={(0,0)}]
  \coordinate (p1) at (0,0); \coordinate (x) at (1.05,0);
  \coordinate (y) at (1.95,0.62); \coordinate (z) at (1.95,-0.62); \coordinate (p2) at (2.85,0);
  \draw[age] (p1)--(x) (x)--(y) (x)--(z) (y)--(z) (y)--(p2) (z)--(p2);
  \node[cut] at (p1) {}; \node[cut] at (p2) {};
  \node[iv] at (x) {}; \node[iv] at (y) {}; \node[iv] at (z) {};
  \node[clab,left=2pt] at (p1) {$p_1$}; \node[clab,right=2pt] at (p2) {$p_2$};
  \node[clab,above=1.5pt] at (x) {$x$}; \node[clab,above=1.5pt] at (y) {$y$}; \node[clab,below=1.5pt] at (z) {$z$};
  \node[ttl,text=aside] at (-0.95,0) {$A$};
\end{scope}

\begin{scope}[shift={(4.7,0)}]
  \coordinate (p1) at (0,0); \coordinate (d) at (0.95,0);
  \coordinate (e) at (1.8,0.62); \coordinate (f) at (1.8,-0.62);
  \coordinate (g) at (2.9,0.62); \coordinate (h) at (2.9,-0.62); \coordinate (p2) at (3.85,0);
  \draw[bge] (p1)--(d) (d)--(e) (d)--(f) (e)--(f) (e)--(g) (f)--(h) (g)--(h) (g)--(p2) (h)--(p2);
  \node[cut] at (p1) {}; \node[cut] at (p2) {};
  \node[iv] at (d) {}; \node[iv] at (e) {}; \node[iv] at (f) {}; \node[iv] at (g) {}; \node[iv] at (h) {};
  \node[clab,left=2pt] at (p1) {$p_1$}; \node[clab,right=2pt] at (p2) {$p_2$};
  \node[clab,above=1.5pt] at (d) {$d$}; \node[clab,above=1.5pt] at (e) {$e$}; \node[clab,below=1.5pt] at (f) {$f$};
  \node[clab,above=1.5pt] at (g) {$g$}; \node[clab,below=1.5pt] at (h) {$h$};
  \node[ttl,text=bside] at (-0.95,0) {$B$};
\end{scope}

\begin{scope}[shift={(1.42,-2.45)}]
  \draw[abead] (129:1.0) arc[start angle=129, end angle=51,  radius=1.0];
  \draw[abead] (39:1.0)  arc[start angle=39,  end angle=-39, radius=1.0];
  \draw[bbead] (-51:1.0) arc[start angle=-51, end angle=-129,radius=1.0];
  \draw[bbead] (-141:1.0)arc[start angle=-141,end angle=-219,radius=1.0];
  \foreach \ang in {45,135,225,315}{ \node[cut] at (\ang:1.0) {}; }
  \node[clab,text=aside] at (90:0.66) {$A$};  \node[clab,text=aside] at (0:0.66) {$A$};
  \node[clab,text=bside] at (-90:0.66) {$B$}; \node[clab,text=bside] at (180:0.66) {$B$};
  \node[ttl] at (0,0) {$G_1$};
\end{scope}

\begin{scope}[shift={(6.62,-2.45)}]
  \draw[abead] (129:1.0) arc[start angle=129, end angle=51,  radius=1.0];
  \draw[bbead] (39:1.0)  arc[start angle=39,  end angle=-39, radius=1.0];
  \draw[abead] (-51:1.0) arc[start angle=-51, end angle=-129,radius=1.0];
  \draw[bbead] (-141:1.0)arc[start angle=-141,end angle=-219,radius=1.0];
  \foreach \ang in {45,135,225,315}{ \node[cut] at (\ang:1.0) {}; }
  \node[clab,text=aside] at (90:0.66) {$A$};  \node[clab,text=bside] at (0:0.66) {$B$};
  \node[clab,text=aside] at (-90:0.66) {$A$}; \node[clab,text=bside] at (180:0.66) {$B$};
  \node[ttl] at (0,0) {$G_2$};
\end{scope}
\end{tikzpicture}}
\end{center}
\caption{The $2$-pole graphs $A$ and $B$ (top) and the graphs $G_1=A,A,B,B$ and $G_2=A,B,A,B$.}
\label{fig:bead-example}
\end{figure}
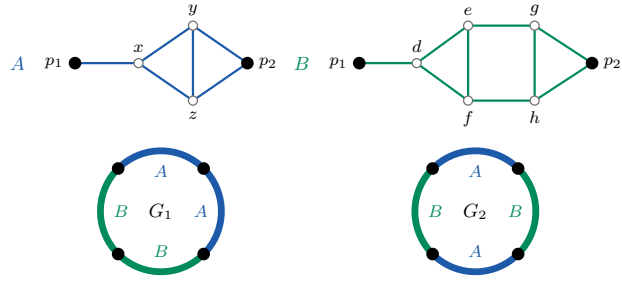

\begin{figure}[H]
\begin{center}
\def\arrfigscale{0.8}%
\scalebox{\arrfigscale}{%
\begin{tikzpicture}[line cap=round, line join=round, >=Latex,
  iv/.style={circle,draw=black!55,fill=white,inner sep=1.4pt,line width=0.6pt},
  cut/.style={circle,fill=black,inner sep=1.9pt},
  ge/.style={line width=1.0pt},
  lensA/.style={dashed,aside!65,fill=aside!14,line width=0.8pt},
  lensB/.style={dashed,bside!65,fill=bside!14,line width=0.8pt},
  cutmark/.style={royalpurple,line width=1.0pt},
  clab/.style={font=\footnotesize},
  ttl/.style={font=\small},
  twist/.style={-{Stealth[length=3mm]},line width=1.2pt,royalpurple}]
\begin{scope}[shift={(0,0)}]
  \foreach \i/\ang in {1/90,2/45,3/0,4/-45,5/-90,6/-135,7/180,8/135}{
    \coordinate (w\i) at (\ang:1.35);
  }
  \draw[ge] (w2)--(w3); \draw[ge] (w4)--(w5); \draw[ge] (w6)--(w7); \draw[ge] (w8)--(w1);
  \path (w1)--(w2) coordinate[midway] (M1);
  \path (w3)--(w4) coordinate[midway] (M3);
  \path (w5)--(w6) coordinate[midway] (M5);
  \path (w7)--(w8) coordinate[midway] (M7);
  \draw[lensA,rotate around={-22.5:(M1)}]  (M1) ellipse [x radius=0.48, y radius=0.165];
  \draw[lensA,rotate around={-112.5:(M3)}] (M3) ellipse [x radius=0.48, y radius=0.165];
  \draw[lensB,rotate around={157.5:(M5)}] (M5) ellipse [x radius=0.48, y radius=0.165];
  \draw[lensB,rotate around={67.5:(M7)}]  (M7) ellipse [x radius=0.48, y radius=0.165];
  \draw[cutmark] (w3)--(w6);
  \draw[cutmark] (w3) circle (0.16); \draw[cutmark] (w6) circle (0.16);
  \draw[royalpurple,dotted,line width=0.9pt] (0:1.98) arc[start angle=0, end angle=-135, radius=1.98];
  \foreach \i in {1,3,5,7}{\node[cut] at (w\i) {};}
  \foreach \i in {2,4,6,8}{\node[iv] at (w\i) {};}
  \node[clab,aside] at (67.5:1.72) {$A'_1$};
  \node[clab,aside] at (-22.5:1.76) {$A'_2$};
  \node[clab,bside] at (-112.5:1.76) {$B'_1$};
  \node[clab,bside] at (157.5:1.72) {$B'_2$};
  \node[ttl] at (0,-2.5) {$G_1$};
\end{scope}
\draw[twist] (2.65,0)--(3.75,0);
\node[clab,royalpurple] at (3.2,0.32) {$\tau$};
\begin{scope}[shift={(6.35,0)}]
  \foreach \i/\ang in {1/90,2/45,3/0,4/-45,5/-90,6/-135,7/180,8/135}{
    \coordinate (w\i) at (\ang:1.35);
  }
  \draw[ge] (w2)--(w3); \draw[ge] (w4)--(w5); \draw[ge] (w6)--(w7); \draw[ge] (w8)--(w1);
  \path (w1)--(w2) coordinate[midway] (M1);
  \path (w3)--(w4) coordinate[midway] (M3);
  \path (w5)--(w6) coordinate[midway] (M5);
  \path (w7)--(w8) coordinate[midway] (M7);
  \draw[lensA,rotate around={-22.5:(M1)}]  (M1) ellipse [x radius=0.48, y radius=0.165];
  \draw[lensB,rotate around={-112.5:(M3)}] (M3) ellipse [x radius=0.48, y radius=0.165];
  \draw[lensA,rotate around={157.5:(M5)}] (M5) ellipse [x radius=0.48, y radius=0.165];
  \draw[lensB,rotate around={67.5:(M7)}]  (M7) ellipse [x radius=0.48, y radius=0.165];
  \foreach \i in {1,3,5,7}{\node[cut] at (w\i) {};}
  \foreach \i in {2,4,6,8}{\node[iv] at (w\i) {};}
  \node[clab,aside] at (67.5:1.72) {$A'_1$};
  \node[clab,bside] at (-22.5:1.76) {$B'_1$};
  \node[clab,aside] at (-112.5:1.76) {$A'_2$};
  \node[clab,bside] at (157.5:1.72) {$B'_2$};
  \node[ttl] at (0,-2.5) {$G_2$};
\end{scope}
\end{tikzpicture}}
\end{center}
\caption{The $S$-node of Remark~\ref{rem:arrangement-needed} in its two arrangements.}
\label{fig:bead-arrangements}
\end{figure}
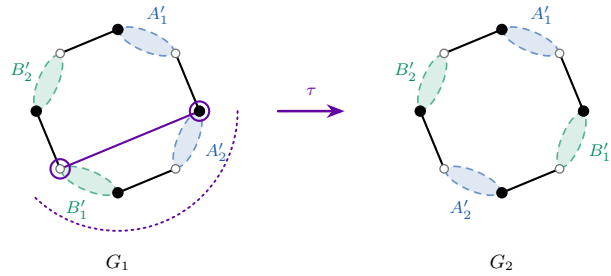

\begin{theorem}\label{thm:cubic-whitney}
Let $G$ and $H$ be finite simple $2$-connected cubic graphs with $M(G)\cong M(H)$. Then there is a sequence $G=G_0, G_1,\dots, G_N\cong H$ in which every $G_i$ is a simple $2$-connected cubic graph and each step $G_{i-1}\to G_i$ is a balanced Whitney twist.
\end{theorem}

\begin{proof}
Fix the isomorphisms of Proposition~\ref{prop:realizations}, so that $G$ and $H$ are realizations of $M$ specified by assembly data on the common decomposition. We convert the data of $G$ into the data of $H$ in two steps. Every twist used is balanced, so every intermediate graph is simple, $2$-connected, and cubic by Lemma~\ref{lem:chain-maintenance}.

\emph{Step 1: arrangements.} Consider the $S$-nodes one at a time. For an $S$-node $\nu$, the data of $H$ at $\nu$ alternates real and virtual members, because $H$ is cubic (Lemma~\ref{lem:snode-structure}(a)). Hence, by Lemma~\ref{lem:snode-rearrange}, finitely many balanced twists across skeleton vertices of $\nu$ carry the current graph to one whose data at $\nu$ is that of $H$, with all other data unchanged. After all $S$-nodes are treated, the current graph agrees with $H$ at every $S$-node and at every tree edge with an $S$-node endpoint.

\emph{Step 2: identifications.} By Step~1, the tree edges at which the current graph and $H$ still disagree have no $S$-node endpoint. Each displays, in the current graph, a Whitney separation with the $M$-determined edge partition (Proposition~\ref{prop:realizations}(i)), and that separation is balanced (Lemma~\ref{lem:snode-structure}(c)). Flip any disagreeing identifications one at a time. Each flip is the balanced Whitney twist across that separation, and it changes that one data value and nothing else (Proposition~\ref{prop:realizations}(ii)).

After the two steps the current graph has the assembly data of $H$, hence it is isomorphic to $H$ by Proposition~\ref{prop:realizations}(ii), and the concatenated twists are the desired sequence.
\end{proof}

\section{Matroid invariance of the genus polynomial}
\label{sec:matroid-invariance}

\begin{theorem}\label{thm:matroid-invariance}
Let $G$ and $H$ be finite simple $2$-connected cubic graphs. If $M(G)\cong M(H)$, then
\[
\Phi_G(y)=\Phi_H(y)\qquad\text{and}\qquad \Gamma_G(x)=\Gamma_H(x).
\]
\end{theorem}

\begin{proof}
By Theorem~\ref{thm:cubic-whitney}, there is a sequence $G=G_0, G_1,\dots, G_N\cong H$ of $2$-connected cubic graphs in which each step $G_{i-1}\to G_i$ is a balanced Whitney twist. By Theorem~\ref{thm:single-twist}, every such step preserves both the genus polynomial and the face-count polynomial, so $\Phi_{G_{i-1}}=\Phi_{G_i}$ and $\Gamma_{G_{i-1}}=\Gamma_{G_i}$ for each $i$. A graph isomorphism preserves $\Phi$ and $\Gamma$ as well, so $\Phi_{G_N}=\Phi_H$ and $\Gamma_{G_N}=\Gamma_H$. Composing the chain of equalities gives $\Phi_G=\Phi_H$ and $\Gamma_G=\Gamma_H$.
\end{proof}

This is Theorem~\ref{thmA}. It holds because a balanced twist induces a face-count-preserving bijection of rotation systems, so the genus polynomial cannot distinguish $2$-isomorphic cubic graphs.

\begin{remark}[Necessity of the cubic hypothesis]\label{rem:cubic-necessary}
The degree restriction in Theorem~\ref{thm:matroid-invariance} is necessary. For graphs that are not cubic, the genus polynomial is in general \emph{not} a cycle-matroid invariant (Figure~\ref{fig:counterexample}). 

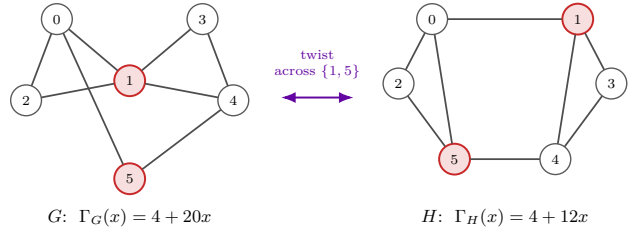
\begin{figure}[htbp]
\centering
\resizebox{0.5\textwidth}{!}{%
\begin{tikzpicture}[>=Latex, line cap=round,
   v/.style={circle,draw=black!65,fill=white,minimum size=5.5mm,inner sep=0pt,font=\scriptsize,line width=0.7pt},
   cut/.style={circle,draw=cutred,fill=cutred!15,minimum size=5.5mm,inner sep=0pt,font=\scriptsize,line width=1pt},
   e/.style={black!70,line width=0.9pt}]
 \begin{scope}[shift={(0,0)}]
   \node[v] (n0) at (-1.3,1.3){0}; \node[cut] (n1) at (0,0.2){1}; \node[v] (n2) at (-1.85,-0.15){2};
   \node[v] (n3) at (1.3,1.3){3}; \node[v] (n4) at (1.85,-0.15){4}; \node[cut] (n5) at (0,-1.55){5};
   \draw[e] (n0)--(n1); \draw[e] (n0)--(n2); \draw[e] (n1)--(n2);
   \draw[e] (n1)--(n3); \draw[e] (n1)--(n4); \draw[e] (n3)--(n4);
   \draw[e] (n0)--(n5); \draw[e] (n4)--(n5);
   \node[font=\small] at (0,-2.25){$G$:\ \ $\Gamma_G(x)=4+20x$};
 \end{scope}
 \draw[<->, royalpurple, line width=1pt] (2.7,-0.1) -- (4.0,-0.1);
 \node[royalpurple, font=\scriptsize, align=center] at (3.35,0.5){twist\\ across $\{1,5\}$};
 \begin{scope}[shift={(6.7,0)}]
   \node[v] (m0) at (-1.3,1.3){0}; \node[cut] (m1) at (1.3,1.3){1}; \node[v] (m2) at (-1.9,0.15){2};
   \node[v] (m3) at (1.9,0.15){3}; \node[v] (m4) at (0.9,-1.2){4}; \node[cut] (m5) at (-0.9,-1.2){5};
   \draw[e] (m0)--(m2); \draw[e] (m0)--(m5); \draw[e] (m2)--(m5);
   \draw[e] (m1)--(m3); \draw[e] (m1)--(m4); \draw[e] (m3)--(m4);
   \draw[e] (m0)--(m1); \draw[e] (m4)--(m5);
   \node[font=\small] at (0,-2.25){$H$:\ \ $\Gamma_H(x)=4+12x$};
 \end{scope}
\end{tikzpicture}
}
\caption{Why the cubic hypothesis is needed in Theorem~\ref{thm:matroid-invariance}.}
\label{fig:counterexample}
\end{figure}

\end{remark}

\part*{Part II. Cospectrality does not distinguish the genus polynomial}
\addcontentsline{toc}{part}{Part II. Cospectrality does not distinguish the genus polynomial}

We now demonstrate that cospectral cubic graphs can have distinct genus polynomials even when they agree on a long list of further invariants. This part exhibits the witnesses (Section~\ref{sec:witnesses}), constructs an infinite family of the same kind (Section~\ref{sec:infinite-family}), measures how often the genus polynomial separates cubic graphs (Section~\ref{sec:census}), explains the separation at the level of the expected face count (Sections~\ref{sec:decomposition}--\ref{sec:longcircuits}), and gives a computable genus bound (Section~\ref{sec:bounds}).

\section{Cospectral cubic graphs with distinct genus polynomials}
\label{sec:witnesses}

\subsection{The smallest cospectral cubic examples}

\begin{theorem}\label{thm:n16-smallest-strong-witness}
There are three connected cubic graphs on $16$ vertices with the same values of
\[
n,\ m,\ \ell,\ \gamma,\ \gmax,\ \operatorname{diam},\ \kappa,\ \lambda,\ |\Aut(G)|,
\]
the same cycle counts $c_3,\dots,c_{10}$, and the same adjacency spectrum (hence, being cubic, the same Laplacian spectrum and the same number of spanning trees) but pairwise distinct orientable genus polynomials. No smaller such triple exists.
\end{theorem}

\begin{proof}
By exhaustive computation over the connected cubic census of orders $8$--$16$ (Section~\ref{sec:census}): the key search and checks are given as item~(D2) of Appendix~\ref{app:data}.
\end{proof}

\begin{table}[htbp]
\centering
\small
\resizebox{0.7\textwidth}{!}{%
\begin{tabular}{@{} l @{\qquad} r@{}r@{}r@{}r @{\qquad} r @{}}
\toprule
\multicolumn{1}{@{}c}{graph6} & \multicolumn{4}{c}{orientable genus polynomial} & \multicolumn{1}{c@{}}{$\Exp[g]$} \\
\midrule
\texttt{O???C@\_UAKECh?DOPO?U?} & $12x$ & ${}+2036x^2$ & ${}+26432x^3$ & ${}+37056x^4$ & $57901/16384\approx 3.5340$ \\
\texttt{O??CA?oID@WAR?E\_Ag?M?} & $20x$ & ${}+2028x^2$ & ${}+26304x^3$ & ${}+37184x^4$ & $57931/16384\approx 3.5358$ \\
\texttt{O?AA@?OaEGKAL?EO@S?F?} & $12x$ & ${}+2148x^2$ & ${}+25872x^3$ & ${}+37504x^4$ & $57985/16384\approx 3.5391$ \\
\bottomrule
\end{tabular}}
\caption{The smallest cospectral examples with distinct genus polynomials among connected cubic graphs. The three graphs agree on the package of Theorem~\ref{thm:n16-smallest-strong-witness} but have distinct orientable genus polynomials.}
\label{tab:n16-smallest-witness}
\end{table}

\subsection{A six-graph family on $22$ vertices}

For the analysis in Section~\ref{sec:longcircuits}, we use a larger family whose differences can be isolated by length.

\begin{theorem}\label{thm:main-witness-class}
There exist six non-isomorphic connected cubic graphs on $22$ vertices agreeing on order, size, girth, diameter, radius, vertex and edge connectivity, automorphism-group order, chromatic number, chromatic index, domination number, independence number, matching number, clique number, circumference, minimum orientable genus, and adjacency spectrum (hence also Laplacian spectrum and spanning-tree count) but whose orientable genus polynomials are pairwise distinct.
\end{theorem}

\begin{proof}
The invariants listed in the statement and the adjacency characteristic polynomial were computed and found equal across the six graphs, and the six orientable genus polynomials, obtained by exhaustive enumeration of the $2^{22}$ rotation systems of each graph (Appendix~\ref{app:data}), are pairwise distinct (Figure~\ref{fig:n22-six-witnesses}). 
\end{proof}

There also exist graphs on as few as 12 vertices with distinct spectra and equal genus polynomials, which we exhibit in Example~\ref{ex:converse} (Section~\ref{sec:synthesis}). Together these exhibit the two directions of Theorems~\ref{thmB} and~\ref{thmB}$'$. This raises a natural question: what embedding structure differs among graphs that agree on all of the invariants displayed above? 

\tikzset{
  wv/.style={circle,draw=black!55,fill=white,inner sep=0pt,minimum size=3pt,line width=0.3pt},
  wskel/.style={black!50,line width=0.5pt,line cap=round},
}
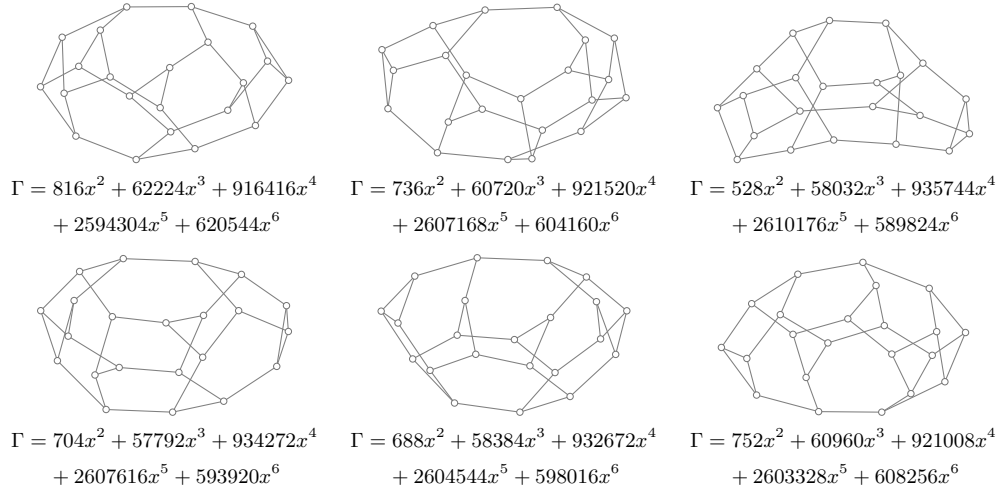
\begin{figure}[htbp]
\centering

\resizebox{0.8\textwidth}{!}{%
\begin{minipage}{\textwidth}
\centering

\begin{subfigure}[t]{0.32\linewidth}
\centering
\begin{tikzpicture}
  \foreach \i/\x/\y in {0/-1.013/0.84, 1/-0.576/1.223, 2/-1.37/0.241, 3/-0.85/0.065, 4/-0.42/-1.3, 5/1.095/-0.485, 6/1.552/-0.738, 7/2.1/0.01, 8/1.752/0.327, 9/1.353/-0.031, 10/0.768/0.64, 11/1.507/0.905, 12/0.551/-1.121, 13/0.151/-0.842, 14/0.135/0.22, 15/-1.417/-0.913, 16/0.492/1.229, 17/-0.531/-0.249, 18/-0.023/-0.443, 19/-2.002/-0.097, 20/-1.612/-0.204, 21/-1.644/0.722} {\coordinate (n\i) at (\x,\y);}
  \foreach \a/\b in {0/1, 0/2, 0/3, 1/16, 1/21, 2/17, 2/19, 3/18, 3/20, 4/12, 4/13, 4/15, 5/8, 5/9, 5/13, 6/7, 6/9, 6/12, 7/8, 7/11, 8/11, 9/10, 10/14, 10/16, 11/16, 12/18, 13/17, 14/17, 14/18, 15/19, 15/20, 19/21, 20/21} {\draw[wskel] (n\a)--(n\b);}
  \foreach \i in {0,...,21} {\node[wv] at (n\i) {};}
\end{tikzpicture}

\caption*{%
  \resizebox{0.96\linewidth}{!}{%
    $\begin{gathered}
      \Gamma=816x^{2}+62224x^{3}+916416x^{4}\\
      {}+2594304x^{5}+620544x^{6}
    \end{gathered}$%
  }%
}
\end{subfigure}
\hfill
\begin{subfigure}[t]{0.32\linewidth}
\centering
\begin{tikzpicture}
  \foreach \i/\x/\y in {0/-1.182/-1.105, 1/-0.018/-1.217, 2/-1.003/-0.597, 3/-2.001/-0.376, 4/0.796/1.3, 5/-0.395/1.255, 6/1.263/0.759, 7/0.98/0.261, 8/1.375/-0.233, 9/1.743/0.79, 10/1.648/0.106, 11/1.936/-0.195, 12/0.196/-0.211, 13/0.552/-0.731, 14/-1.256/0.998, 15/-1.046/0.504, 16/0.38/-1.204, 17/-0.442/-0.385, 18/-0.706/0.177, 19/1.191/-0.752, 20/-1.911/0.258, 21/-2.1/0.599} {\coordinate (n\i) at (\x,\y);}
  \foreach \a/\b in {0/1, 0/2, 0/3, 1/16, 1/19, 2/17, 2/18, 3/20, 3/21, 4/5, 4/6, 4/9, 5/14, 5/15, 6/7, 6/8, 7/10, 7/12, 8/11, 8/13, 9/10, 9/11, 10/19, 11/19, 12/16, 12/18, 13/16, 13/17, 14/18, 14/21, 15/17, 15/20, 20/21} {\draw[wskel] (n\a)--(n\b);}
  \foreach \i in {0,...,21} {\node[wv] at (n\i) {};}
\end{tikzpicture}

\caption*{%
  \resizebox{0.96\linewidth}{!}{%
    $\begin{gathered}
      \Gamma=736x^{2}+60720x^{3}+921520x^{4}\\
      {}+2607168x^{5}+604160x^{6}
    \end{gathered}$%
  }%
}
\end{subfigure}
\hfill
\begin{subfigure}[t]{0.32\linewidth}
\centering
\begin{tikzpicture}
  \foreach \i/\x/\y in {0/-1.416/0.501, 1/-0.702/-0.198, 2/-0.819/1.09, 3/-2.067/-0.143, 4/0.886/-0.746, 5/-0.144/-0.676, 6/0.96/0.394, 7/1.766/-0.856, 8/2.044/0.002, 9/2.1/-0.573, 10/1.33/0.594, 11/1.283/-0.27, 12/0.717/1.255, 13/0.57/0.272, 14/-0.772/0.354, 15/-0.855/-0.842, 16/0.501/-0.123, 17/-0.322/0.211, 18/-0.214/1.3, 19/-1.74/-0.999, 20/-1.641/0.058, 21/-1.462/-0.605} {\coordinate (n\i) at (\x,\y);}
  \foreach \a/\b in {0/1, 0/2, 0/3, 1/16, 1/21, 2/17, 2/18, 3/19, 3/20, 4/5, 4/6, 4/7, 5/14, 5/15, 6/12, 6/13, 7/8, 7/9, 8/9, 8/10, 9/11, 10/12, 10/16, 11/13, 11/16, 12/18, 13/17, 14/18, 14/20, 15/17, 15/19, 19/21, 20/21} {\draw[wskel] (n\a)--(n\b);}
  \foreach \i in {0,...,21} {\node[wv] at (n\i) {};}
\end{tikzpicture}

\caption*{%
  \resizebox{0.96\linewidth}{!}{%
    $\begin{gathered}
      \Gamma=528x^{2}+58032x^{3}+935744x^{4}\\
      {}+2610176x^{5}+589824x^{6}
    \end{gathered}$%
  }%
}
\end{subfigure}

\par\vspace{0.6em}

\begin{subfigure}[t]{0.32\linewidth}
\centering
\begin{tikzpicture}
  \foreach \i/\x/\y in {0/2.069/0.461, 1/1.324/0.978, 2/2.1/0.034, 3/1.908/-0.543, 4/0.683/-0.392, 5/-1.352/1.026, 6/-1.443/0.546, 7/0.186/-1.3, 8/-0.811/0.279, 9/-1.098/-0.685, 10/-0.697/-0.561, 11/-0.914/-1.256, 12/-1.995/0.378, 13/-1.719/-0.449, 14/-1.541/-0.043, 15/0.29/-0.642, 16/0.076/0.177, 17/-0.63/1.239, 18/0.558/1.192, 19/0.697/0.3, 20/1.031/-1.119, 21/1.278/0.378} {\coordinate (n\i) at (\x,\y);}
  \foreach \a/\b in {0/1, 0/2, 0/3, 1/18, 1/19, 2/3, 2/21, 3/20, 4/7, 4/16, 4/21, 5/8, 5/12, 5/17, 6/13, 6/14, 6/17, 7/11, 7/20, 8/9, 8/16, 9/10, 9/11, 10/14, 10/15, 11/13, 12/13, 12/14, 15/19, 15/20, 16/19, 17/18, 18/21} {\draw[wskel] (n\a)--(n\b);}
  \foreach \i in {0,...,21} {\node[wv] at (n\i) {};}
\end{tikzpicture}

\caption*{%
  \resizebox{0.96\linewidth}{!}{%
    $\begin{gathered}
      \Gamma=704x^{2}+57792x^{3}+934272x^{4}\\
      {}+2607616x^{5}+593920x^{6}
    \end{gathered}$%
  }%
}
\end{subfigure}
\hfill
\begin{subfigure}[t]{0.32\linewidth}
\centering
\begin{tikzpicture}
  \foreach \i/\x/\y in {0/1.294/0.935, 1/0.648/1.207, 2/0.701/0.265, 3/1.962/0.376, 4/-0.887/-1.149, 5/0.194/-1.3, 6/-1.579/-0.42, 7/-1.292/-0.609, 8/-2.1/0.37, 9/-1.822/0.174, 10/-0.713/0.546, 11/-1.544/0.95, 12/-0.843/-0.027, 13/-0.536/-0.343, 14/0.775/-0.645, 15/1.024/-1.041, 16/-0.512/1.253, 17/0.36/-0.534, 18/0.103/-0.1, 19/1.778/-0.345, 20/1.528/-0.089, 21/1.46/0.527} {\coordinate (n\i) at (\x,\y);}
  \foreach \a/\b in {0/1, 0/2, 0/3, 1/16, 1/21, 2/17, 2/18, 3/19, 3/20, 4/5, 4/6, 4/7, 5/14, 5/15, 6/8, 6/12, 7/9, 7/13, 8/9, 8/11, 9/11, 10/12, 10/13, 10/16, 11/16, 12/18, 13/17, 14/18, 14/20, 15/17, 15/19, 19/21, 20/21} {\draw[wskel] (n\a)--(n\b);}
  \foreach \i in {0,...,21} {\node[wv] at (n\i) {};}
\end{tikzpicture}

\caption*{%
  \resizebox{0.96\linewidth}{!}{%
    $\begin{gathered}
      \Gamma=688x^{2}+58384x^{3}+932672x^{4}\\
      {}+2604544x^{5}+598016x^{6}
    \end{gathered}$%
  }%
}
\end{subfigure}
\hfill
\begin{subfigure}[t]{0.32\linewidth}
\centering
\begin{tikzpicture}
  \foreach \i/\x/\y in {0/-0.836/1.104, 1/0.24/1.3, 2/-1.126/0.436, 3/-1.599/0.617, 4/-0.492/-1.168, 5/0.552/-1.177, 6/1.034/-0.871, 7/1.387/-0.237, 8/0.718/-0.216, 9/1.596/-0.674, 10/1.46/0.154, 11/1.937/0.14, 12/-0.009/0.366, 13/0.594/0.254, 14/-0.699/-0.603, 15/-1.518/-0.894, 16/0.457/0.917, 17/-0.339/-0.032, 18/-0.916/0.11, 19/1.337/0.871, 20/-1.679/-0.289, 21/-2.1/-0.107} {\coordinate (n\i) at (\x,\y);}
  \foreach \a/\b in {0/1, 0/2, 0/3, 1/16, 1/19, 2/17, 2/20, 3/18, 3/21, 4/5, 4/14, 4/15, 5/6, 5/9, 6/7, 6/8, 7/11, 7/13, 8/10, 8/12, 9/10, 9/11, 10/19, 11/19, 12/16, 12/18, 13/16, 13/17, 14/17, 14/18, 15/20, 15/21, 20/21} {\draw[wskel] (n\a)--(n\b);}
  \foreach \i in {0,...,21} {\node[wv] at (n\i) {};}
\end{tikzpicture}

\caption*{%
  \resizebox{0.96\linewidth}{!}{%
    $\begin{gathered}
      \Gamma=752x^{2}+60960x^{3}+921008x^{4}\\
      {}+2603328x^{5}+608256x^{6}
    \end{gathered}$%
  }%
}
\end{subfigure}

\end{minipage}%
}

\caption{The six connected cubic graphs on $22$ vertices of
Theorem~\ref{thm:main-witness-class}, labeled $W_1,\dots,W_6$ left to right and
top to bottom, with their orientable genus polynomials.}
\label{fig:n22-six-witnesses}
\end{figure}
\section{Godsil--McKay switching and an infinite family}
\label{sec:infinite-family}

The witnesses of Section~\ref{sec:witnesses} were located by search, which leaves open whether such pairs are only small-order examples. In each family exhibited above, iterated \emph{Godsil--McKay switches} connect any two members (Proposition~\ref{prop:witnesses-are-gm}). We use this to construct an explicit infinite family of cospectral cubic pairs whose orientable genus polynomials differ, and differ already at their minimum genus.

\subsection{Godsil--McKay switching}

\begin{definition}\label{def:gm-switch}
Let $X$ be a graph and $C\subseteq V(X)$ with $|C|$ even, such that the subgraph of $X$ induced on $C$ is regular and every vertex $v\notin C$ has $0$, $|C|/2$, or $|C|$ neighbors in $C$. The \emph{Godsil--McKay switch} $\operatorname{switch}_C(X)$ is obtained by replacing, for each $v\notin C$ with exactly $|C|/2$ neighbors in $C$, the edges between $v$ and $C$ by the complementary set, leaving all other edges unchanged.
\end{definition}

\begin{proposition}[Godsil--McKay~{\cite[Theorem~2.2]{GodsilMcKay}}]\label{prop:gm-cospectral}
$\operatorname{switch}_C(X)$ is adjacency-cospectral with $X$.
\end{proposition}

As we have seen, cospectral cubic graphs are automatically Laplacian-cospectral as well. Thus, Godsil--McKay switching between cubic graphs preserves both spectra in the invariant package of Section~\ref{sec:witnesses}. 
\begin{proposition}\label{prop:witnesses-are-gm}
Within each family of Theorems~\ref{thm:n16-smallest-strong-witness} and~\ref{thm:main-witness-class}, any two graphs are related by a chain of Godsil--McKay switches at four-element switching sets, through members of the family. Eight of the eighteen cospectral pairs (two of the three at $n=16$, six of the fifteen at $n=22$) are related by a single switch, and these single switches connect each family, any two members being at most two ($n=16$), respectively four ($n=22$), switches apart. The remaining ten pairs are related by no single switch. 
\end{proposition}

\begin{proof}[Verification]
For each of the eight single-switch pairs, a four-element independent set $C$ meeting Definition~\ref{def:gm-switch}, together with a labeling under which $\operatorname{switch}_C$ carries one graph to the other, is recorded in the supplementary data (Appendix~\ref{app:data}, item~(D3)).
\end{proof}

Proposition~\ref{prop:witnesses-are-gm} also explains why such graphs may differ in genus at all. A Godsil--McKay switch is \emph{not} a Whitney twist, and we see that it may alter the cycle matroid (Remark~\ref{rem:gm-matroid}).

\subsection{The graph $A$ and the families}

Let $A$ be the connected cubic graph on $V(A)=\{0,1,\dots,13\}$ with adjacency lists
\[
\begin{array}{c|c@{\qquad}c|c@{\qquad}c|c@{\qquad}c|c}
v & N_A(v) & v & N_A(v) & v & N_A(v) & v & N_A(v)\\ \hline
0&\{6,8,9\} & 4&\{10,12,13\} & 8&\{0,6,7\} & 12&\{2,4,5\}\\
1&\{7,9,10\} & 5&\{10,12,13\} & 9&\{0,1,3\} & 13&\{3,4,5\}\\
2&\{7,11,12\}& 6&\{0,8,11\} & 10&\{1,4,5\} & &\\
3&\{9,11,13\}& 7&\{1,2,8\} & 11&\{2,3,6\} & &
\end{array}
\]
equivalently, up to a relabeling of the vertices, $A$ has graph6 string \verb|M{CgOH@?O`?@?d?d?|. Put
\[
C=\{1,2,3,4\},\qquad W=\{7,9,10,11,12,13\}.
\]
Each $w\in W$ has exactly two neighbors in $C$, while each of $0,5,6,8$ has none, and $C$ is independent, hence induces a $0$-regular subgraph. So $C$ satisfies Definition~\ref{def:gm-switch}.

\begin{definition}[The families $G_t$ and $H_t$]\label{def:GtHt}
For $t\ge1$, introduce new vertices $a_1,\dots,a_t,c_1,\dots,c_t$. Form $G_t$ from $A$ by deleting the edges $\{0,8\}$ and $\{6,11\}$, inserting the two paths
\[
0\,a_1\,a_2\cdots a_t\,8 \qquad\text{and}\qquad 6\,c_1\,c_2\cdots c_t\,11,
\]
and adding the edges $a_ic_i$ for $1\le i\le t$. Define $H_t=\operatorname{switch}_C(G_t)$.
\end{definition}

The two deleted edges have no endpoint in $C$, so replacing them changes no vertex's set of neighbors in $C$, and the inserted vertices have no neighbors in $C$. Hence $C$ is still independent and satisfies Definition~\ref{def:gm-switch} in $G_t$, and $H_t$ differs from $G_t$ only in the edges between $C$ and $W$. Both graphs have $14+2t$ vertices. Figure~\ref{fig:gm-families} draws the first two pairs $G_1,H_1$ and $G_2,H_2$ in full and shows how $G_t$ and $H_t$ grow with $t$.

\begin{figure}[htbp]
\centering

\newcommand\GMcoords{%
  \coordinate (s1) at (0,3.6);    \coordinate (s2) at (0,2.4);
  \coordinate (s3) at (0,1.2);    \coordinate (s4) at (0,0);
  \coordinate (w7) at (2.3,3.6);  \coordinate (w9) at (2.3,2.88);
  \coordinate (w10) at (2.3,2.16);\coordinate (w11) at (2.3,1.44);
  \coordinate (w12) at (2.3,0.72);\coordinate (w13) at (2.3,0);
  \coordinate (n8) at (4.3,3.95); \coordinate (n0) at (4.95,2.8);
  \coordinate (n6) at (4.8,1.65); \coordinate (n5) at (5.15,-0.2);
  \coordinate (a1) at (5.85,3.25);\coordinate (c1) at (5.85,1.55);
  \coordinate (a2) at (6.95,3.05);\coordinate (c2) at (6.95,1.75);}
\newcommand\GMgray{%
  \draw[ggray] (n0)--(n6) (n0)--(w9) (n0)--(a1) (n6)--(n8) (n6)--(c1)
               (n8)--(w7) (w10)--(n5) (w12)--(n5) (w13)--(n5) (a1)--(c1);}
\newcommand\GMorangeG{%
  \draw[gcw] (w9)--(s1) (w9)--(s3) (s1)--(w7) (s1)--(w10) (w7)--(s2)
             (w10)--(s4) (s2)--(w11) (s2)--(w12) (w11)--(s3) (w12)--(s4)
             (s3)--(w13) (w13)--(s4);}
\newcommand\GMorangeH{%
  \draw[gcw] (w9)--(s2) (w9)--(s4) (s1)--(w11) (s1)--(w12) (s1)--(w13)
             (w7)--(s3) (w7)--(s4) (w10)--(s2) (w10)--(s3) (s2)--(w13)
             (w11)--(s4) (w12)--(s3);}
\newcommand\GMnodes{%
  \node[cv] at (s1) {$1$};  \node[cv] at (s2) {$2$};
  \node[cv] at (s3) {$3$};  \node[cv] at (s4) {$4$};
  \node[wv] at (w7) {$7$};  \node[wv] at (w9) {$9$};
  \node[wv] at (w10) {$10$};\node[wv] at (w11) {$11$};
  \node[wv] at (w12) {$12$};\node[wv] at (w13) {$13$};
  \node[nv] at (n0) {$0$};  \node[nv] at (n5) {$5$};
  \node[nv] at (n6) {$6$};  \node[nv] at (n8) {$8$};
  \node[nv] at (a1) {$a_1$};\node[nv] at (c1) {$c_1$};}
\newcommand\GMrowone{%
  \draw[ggray] (n8) to[bend left=34] (a1);
  \draw[ggray] (w11) to[bend right=34] (c1);}
\newcommand\GMrowtwo{%
  \draw[gnew] (n8) to[bend left=34] (a2);
  \draw[gnew] (w11) to[bend right=34] (c2);
  \draw[gnew] (a1)--(a2) (a2)--(c2) (c1)--(c2);
  \node[nvnew] at (a2) {$a_2$}; \node[nvnew] at (c2) {$c_2$};}
\resizebox{0.75\textwidth}{!}{%
\begin{tikzpicture}[line cap=round, line join=round, >=Latex,
  cv/.style={draw=Ccol,fill=Ccol!18,rectangle,minimum size=4.6mm,inner sep=0pt,font=\scriptsize,line width=0.7pt},
  wv/.style={draw=aside,fill=aside!15,circle,minimum size=4.6mm,inner sep=0pt,font=\scriptsize,line width=0.7pt},
  nv/.style={draw=black!55,fill=white,circle,minimum size=4.6mm,inner sep=0pt,font=\scriptsize,line width=0.7pt},
  nvnew/.style={draw=suggestgreen,fill=suggestgreen!18,circle,minimum size=4.6mm,inner sep=0pt,font=\scriptsize,line width=1.1pt},
  ggray/.style={gray!55,line width=0.8pt},
  gcw/.style={Ccol,line width=1.3pt},
  gnew/.style={suggestgreen,line width=1.7pt},
  gname/.style={font=\small},
  swit/.style={-{Stealth[length=2.6mm]},line width=1.1pt,black!70},
  ext/.style={-{Stealth[length=2.6mm]},line width=1.1pt,suggestgreen!85!black}]
\begin{scope}[shift={(0,5.8)}]
  \node[gname] at (3.0,3.85) {$G_1$\,\ ($\gamma=2$)};
  \begin{scope}[yscale=0.86]
    \GMcoords \GMgray \GMrowone \GMorangeG \GMnodes
  \end{scope}
\end{scope}
\begin{scope}[shift={(8.9,5.8)}]
  \node[gname] at (3.0,3.85) {$H_1$\,\ ($\gamma=1$)};
  \begin{scope}[yscale=0.86]
    \GMcoords \GMgray \GMrowone \GMorangeH \GMnodes
  \end{scope}
\end{scope}
\begin{scope}[shift={(0,0)}]
  \node[gname] at (3.0,3.85) {$G_2$\,\ ($\gamma=2$)};
  \begin{scope}[yscale=0.86]
    \GMcoords \GMgray \GMrowtwo \GMorangeG \GMnodes
  \end{scope}
\end{scope}
\begin{scope}[shift={(8.9,0)}]
  \node[gname] at (3.0,3.85) {$H_2$\,\ ($\gamma=1$)};
  \begin{scope}[yscale=0.86]
    \GMcoords \GMgray \GMrowtwo \GMorangeH \GMnodes
  \end{scope}
\end{scope}
  \draw[swit] (6.6,7.3)--(8.3,7.3);  \node[font=\scriptsize,anchor=south] at (7.45,7.4) {$\operatorname{switch}_C$};
  \draw[swit] (7.3,1.5)--(8.6,1.5);  \node[font=\scriptsize,anchor=south] at (7.95,1.6) {$\operatorname{switch}_C$};
  \draw[ext] (3.0,5.3)--(3.0,4.3);   \node[font=\scriptsize,align=left,anchor=west,text width=3.9cm] at (3.25,4.8) {add $a_2,c_2$: subdivide $\{a_1,8\}$, $\{c_1,11\}$ and join $a_2c_2$ (\textcolor{suggestgreen}{green})};
  \draw[ext] (11.9,5.3)--(11.9,4.3); \node[font=\scriptsize,anchor=west] at (12.15,4.8) {(same two steps)};
\end{tikzpicture}}
\caption{The first two members of the family, drawn in full.}
\label{fig:gm-families}
\end{figure}
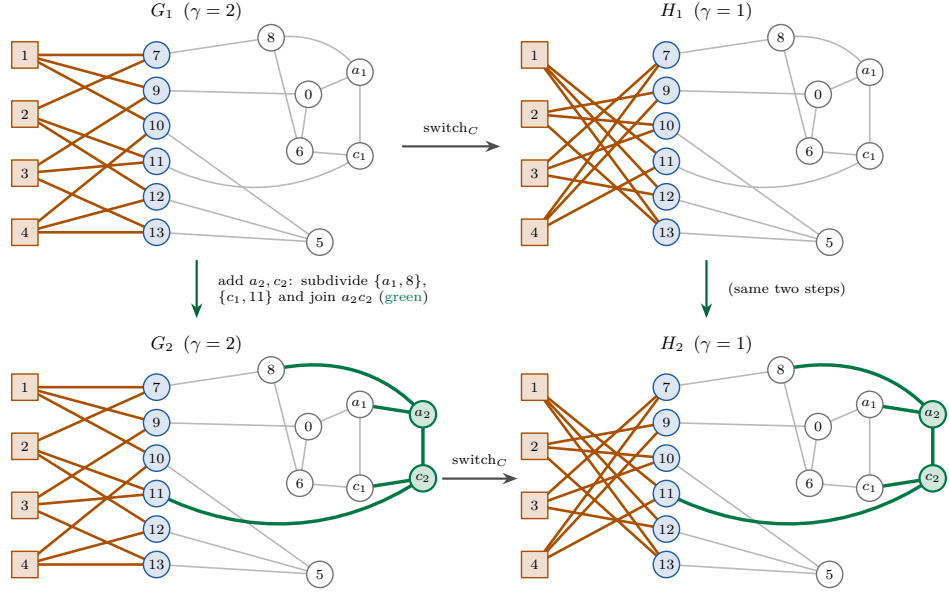

\subsection{The infinite family}

\begin{theorem}\label{thm:infinite-family}
For every $t\ge1$ the graphs $G_t$ and $H_t$ of Definition~\ref{def:GtHt} are connected simple cubic graphs on $14+2t$ vertices, are adjacency-cospectral, and satisfy
\[
\gamma(G_t)=2,\qquad \gamma(H_t)=1.
\]
Consequently $\Gamma_{G_t}(x)\ne\Gamma_{H_t}(x)$ for every $t$. The pairs $(G_t,H_t)_{t\ge1}$ form an infinite family of adjacency-cospectral pairs of connected cubic graphs with distinct orientable genus polynomials.
\end{theorem}

This is Theorem~\ref{thmD}. We prove it through five lemmas.

\begin{lemma}\label{lem:if-structure}
For every $t\ge1$, $G_t$ and $H_t$ are simple, $2$-connected, and cubic.
\end{lemma}

\begin{proof}
Switching replaces each $w\in W$'s two $C$-neighbors by the other two, and each
$c\in C$'s three $W$-neighbors by the complementary three, so
$A':=\operatorname{switch}_C(A)$ is again simple and cubic. The deleted edges $\{0,8\}$
and $\{6,11\}$ have no endpoint in $C$, so $H_t$ is obtained from $A'$ by the
same construction that produces $G_t$ from $A$. In either graph each of
$0,8,6,11$ trades one old edge for one new one, and each inserted vertex gets two
path edges and one rung, so $G_t$ and $H_t$ are simple and cubic. Both base
graphs are Hamiltonian through the two deleted edges:
\[
2\;12\;4\;10\;5\;13\;3\;11\;6\;8\;0\;9\;1\;7\;2 \quad\text{in } A,
\qquad
0\;9\;4\;7\;3\;12\;5\;10\;2\;13\;1\;11\;6\;8\;0 \quad\text{in } A'.
\]
Replacing $\{0,8\}$ and $\{6,11\}$ by the inserted paths turns these into
Hamiltonian cycles of $G_t$ and $H_t$, so both graphs are Hamiltonian and in
particular $2$-connected.
\end{proof}

\begin{lemma}\label{lem:if-cospectral}
For every $t\ge1$, $G_t$ and $H_t$ are adjacency-cospectral.
\end{lemma}

\begin{proof}
The set $C$ satisfies Definition~\ref{def:gm-switch} in $G_t$ just as it does
in $A$. The edge deletions and insertions producing $G_t$ touch no edge at $C$,
and the new vertices $a_i,c_i$ have no neighbors in $C$. Since
$H_t=\operatorname{switch}_C(G_t)$, Proposition~\ref{prop:gm-cospectral} gives
the claim.
\end{proof}

\begin{lemma}[Base case]\label{lem:if-base}
The graphs $G_1$ and $H_1$ have $16$ vertices and orientable genus polynomials
\[
\Gamma_{G_1}(x)=2432\,x^2+29824\,x^3+33280\,x^4,
\]
\[
\Gamma_{H_1}(x)=48\,x+2448\,x^2+28736\,x^3+34304\,x^4.
\]
In particular $\gamma(G_1)=2$ and $\gamma(H_1)=1$. Moreover, $G_1$ admits a genus-$2$ embedding with a face whose boundary contains the arc
\[
(2,7),(7,8),(8,a_1),(a_1,c_1),(c_1,11),(11,2),
\]
and $H_1$ admits a genus-$1$ embedding with a face whose boundary contains
\[
(4,11),(11,c_1),(c_1,a_1),(a_1,8),(8,6).
\]
\end{lemma}

\begin{proof}
Exhaustive enumeration of all $2^{16}$ rotation systems of each graph, tracing the face permutation and computing the genus by Lemma~\ref{lem:genus-from-faces}. Explicit rotation systems realizing the two displayed embeddings can be found in Appendix~\ref{app:data}, items~(D2) and~(D4).
\end{proof}

\begin{lemma}[Upper bound]\label{lem:if-upper}
For every $t\ge1$, $\gamma(G_t)\le2$ and $\gamma(H_t)\le1$.
\end{lemma}

\begin{proof}
We prove the statement by induction on $t$. Lemma~\ref{lem:if-base} gives a genus-$2$ embedding of $G_1$
with a face whose boundary contains
\[
(2,7),(7,8),(8,a_1),(a_1,c_1),(c_1,11),(11,2).
\]
Suppose $G_t$ has such an embedding, with $a_t,c_t$ in place of $a_1,c_1$.
The graph $G_{t+1}$ arises from $G_t$ by subdividing $\{a_t,8\}$ with $a_{t+1}$,
subdividing $\{c_t,11\}$ with $c_{t+1}$, and adding the rung $a_{t+1}c_{t+1}$.
Subdivision changes no face except by lengthening its boundary walk, so the face
above now runs $\dots,7,8,a_{t+1},a_t,c_t,c_{t+1},11,2,\dots$, and drawing the
rung as a chord across this disk face splits it into two: the quadrilateral
$a_{t+1}a_tc_tc_{t+1}$, and a face again carrying the arc, now through
$a_{t+1},c_{t+1}$. Altogether $\Delta V=2$, $\Delta E=3$, $\Delta F=1$, so the
Euler characteristic, hence the genus, is unchanged, and the induction
continues: $\gamma(G_t)\le2$ for every $t$. Since switching touches no
inserted edge, $H_{t+1}$ is built from $H_t$ by the same three operations, and
the same induction from the genus-$1$ embedding of $H_1$, with the chord drawn
so that the surviving face carries
$(4,11),(11,c_{t+1}),(c_{t+1},a_{t+1}),(a_{t+1},8),(8,6)$, gives
$\gamma(H_t)\le1$.
\end{proof}

\begin{lemma}[Lower bound]\label{lem:if-lower}
For every $t\ge1$, $\gamma(G_t)\ge2$ and $\gamma(H_t)\ge1$.
\end{lemma}

\begin{proof}
We show $G_t$ is a minor of $G_{t+1}$. In $G_{t+1}$ delete the last edge $a_{t+1}c_{t+1}$, then contract $a_ta_{t+1}$ and $c_tc_{t+1}$. The contractions restore the edges $\{a_t,8\}$ and $\{c_t,11\}$, and no multi-edge appears because the only would-be duplicate, the edge $a_{t+1}c_{t+1}$, was deleted first. The result is $G_t$. Switching alters none of the inserted edges, so the same operations reduce $H_{t+1}$ to $H_t$. Iterating gives $G_1\preccurlyeq G_t$ and $H_1\preccurlyeq H_t$. Orientable genus is minor-monotone (see, e.g., \cite{MoharThomassen}), so $\gamma(G_t)\ge\gamma(G_1)=2$ and $\gamma(H_t)\ge\gamma(H_1)=1$ by Lemma~\ref{lem:if-base}.
\end{proof}

\begin{proof}[Proof of Theorem~\ref{thm:infinite-family}]
Lemma~\ref{lem:if-structure} gives the structural claims and Lemma~\ref{lem:if-cospectral} cospectrality. Lemmas~\ref{lem:if-upper} and~\ref{lem:if-lower} give $\gamma(G_t)=2$ and $\gamma(H_t)=1$. Since the least exponent of $\Gamma_G$ equals $\gamma(G)$, the polynomial $\Gamma_{G_t}$ has no $x^1$ term while $\Gamma_{H_t}$ does, so $\Gamma_{G_t}\ne\Gamma_{H_t}$ and in particular $G_t\not\cong H_t$. 
\end{proof}

\begin{remark}\label{rem:gm-matroid}
We return to this in Section~\ref{sec:synthesis}. Each $G_t$ and $H_t$ is a finite simple $2$-connected cubic graph (Lemma~\ref{lem:if-structure}), so by Theorem~\ref{thmA} their distinct genus polynomials force $M(G_t)\not\cong M(H_t)$. This shows the Godsil--McKay switch changes the cycle matroid. It is exactly because the switch is \emph{not} a Whitney twist that it is free to change the genus polynomial.
\end{remark}

\section{A census of genus polynomials}
\label{sec:census}

To measure how common the coincidences of Section~\ref{sec:witnesses} are, we
computed the orientable genus polynomial of every connected cubic graph
through $22$ vertices, generated with \texttt{nauty}/\texttt{geng}~\cite{McKayPiperno}.
Brinkmann's \texttt{multi\_genus}~\cite{BrinkmannGenus} previously computed the \emph{minimum} genus of the graphs in this range.

\begin{theorem}\label{thm:census-through-22}
The orientable genus polynomial was computed exactly for all connected cubic graphs on $n\le 22$ vertices. The counts of graphs and of distinct genus polynomials are those of Table~\ref{tab:census}.
\end{theorem}

\begin{proof}[Computational proof]
Exhaustive enumeration of the $2^n$ rotation systems per graph over the connected cubic census.
\end{proof}

\begin{table}[htbp]
\centering
\setlength{\tabcolsep}{12pt}
\renewcommand{\arraystretch}{1.15}
\begin{tabular}{@{}crrr@{}}
\toprule
$n$ & $g_n$ & $p_n$ & $p_n/g_n$ \\
\midrule
$4$  & $1$             & $1$             & $100.0\%$ \\
$6$  & $2$             & $2$             & $100.0\%$ \\
$8$  & $5$             & $5$             & $100.0\%$ \\
$10$ & $19$            & $19$            & $100.0\%$ \\
$12$ & $85$            & $70$            & $82.4\%$ \\
$14$ & $509$           & $387$           & $76.0\%$ \\
$16$ & $4{,}060$       & $2{,}869$       & $70.7\%$ \\
$18$ & $41{,}301$      & $29{,}734$      & $72.0\%$ \\
$20$ & $510{,}489$     & $390{,}097$     & $76.4\%$ \\
$22$ & $7{,}319{,}447$ & $5{,}950{,}365$ & $81.3\%$ \\
\bottomrule
\end{tabular}
\caption{Orientable genus polynomials of connected cubic graphs through $22$ vertices: $g_n$ is the \# of connected cubic graphs and $p_n$ the \# of distinct orientable genus polynomials they realize.}
\label{tab:census}
\end{table}

The genus polynomial is not a complete invariant, and the distinct-polynomial counts of Table~\ref{tab:census} demonstrate this. But the polynomial is also far from redundant, as the witnesses of Section~\ref{sec:witnesses} show.

\section{Faces by length: the short/long decomposition}
\label{sec:decomposition}

The witnesses show that cospectral cubic graphs can have different genus polynomials. This section shows how the difference arises. The genus polynomial is a distribution over faces, and we work with the expected number of faces of a uniformly random rotation system, which by Lemma~\ref{lem:genus-from-faces} controls the average genus through $\Exp[g]=1+\tfrac n4-\tfrac12\Exp[F]$.

The random-embedding viewpoint originates with Stahl, whose bound $\Exp[F]\le n\ln n$ on the expected number of faces~\cite{StahlRegions} was recently sharpened to a linear bound by Campion Loth and Mohar~\cite{CampionLothMohar} and, for random (multi)graphs with given degree sequences of minimum degree at least~$2$, to a logarithmic one~\cite{CampionLothLog}. For connected simple graphs of minimum degree at least~$2$ and maximum degree at most~$3$, Chen and Gao established the linear bound $\Exp[F]\le\tfrac n3+1$~\cite{ChenGao}. In the other direction, any family $\mathcal C$ of cycles, none lying entirely on vertices of degree~$2$, gives a lower bound $\Exp[F]\ge\sum_{C\in\mathcal C}2\prod_{u\in V(C)}(\deg u-1)^{-1}$~\cite[proof of Proposition~13]{CampionLothMultistars}. The decomposition we provide makes the short-cycle contribution to $\Exp[F]$ \emph{exact} and isolates the portion determined by the adjacency spectrum.

\begin{remark}\label{rem:short-faces-simple-cycles}
Let $G$ be a connected cubic graph of girth $\ell$. Every facial boundary walk of length $k<2\ell$ is an oriented traversal of a simple cycle of $G$. A facial walk never backtracks, since it enters $v$ on $(u,v)$ and leaves on $(v,\sigma_v(u))$ with $\sigma_v(u)\neq u$. If it repeated a vertex, the two visits would split it into closed segments with lengths summing to $k$, and a shortest segment over all repeats has distinct interior vertices and length at least $3$, so it is a simple cycle of length at most $k/2<\ell$, contradicting the girth.
\end{remark}

Define the \emph{expected short-face contribution} $S(G):=\Exp[F_{<2\ell}]$, the expected number of faces of length less than $2\ell$. Each simple $k$-cycle is realized as a face with a probability depending only on $k$, so linearity of expectation yields a closed-form expression for $S(G)$.

\begin{corollary}\label{cor:short-face-expectation}
For a connected cubic graph $G$ of girth $\ell$,
\[
S(G)=\sum_{k=\ell}^{2\ell-1}2^{1-k}c_k(G).
\]
\end{corollary}

\begin{proof}
Fix $\ell\le k<2\ell$. By Remark~\ref{rem:short-faces-simple-cycles}, the
length-$k$ faces of $\rho$ are exactly the oriented simple $k$-cycles of $G$
realized as cycles of $\phi_\rho$, and $G$ has $2c_k(G)$ of them. An oriented
cycle is realized iff at each of its vertices the rotation continues the walk
along the cycle, and at a cubic vertex exactly one of the two rotations does. By linearity of expectation,
$\Exp[F_{=k}]=2c_k(G)\,2^{-k}=2^{1-k}c_k(G)$. Since no face is shorter than the
girth, summing over $\ell\le k<2\ell$ gives the formula.
\end{proof}

The per-cycle face probability behind this identity is standard and was computed for general graphs by Campion Loth, Hal\'asz, Masa\v{r}\'ik, Mohar, and \v{S}\'amal~\cite[proof of Proposition~13]{CampionLothMultistars}.

\begin{definition}\label{def:long-face-contribution}
The \emph{expected long-face contribution} $L(G):=\Exp[F_{\ge2\ell}]$ is the expected number of faces of length at least $2\ell$.
\end{definition}

\begin{proposition}[Short/long decomposition]\label{prop:short-long-decomposition}
For a connected cubic graph $G$ of girth $\ell$,
\[
\Exp[F]=S(G)+L(G),\qquad S(G)=\sum_{k=\ell}^{2\ell-1}2^{1-k}c_k(G).
\]
\end{proposition}

\begin{proof}
Each face has a length, so $F(\rho)=F_{<2\ell}(\rho)+F_{\ge2\ell}(\rho)$ for every $\rho$. Taking expectations,
\[
\Exp[F]=\Exp[F_{<2\ell}]+\Exp[F_{\ge2\ell}]=S(G)+L(G).
\]
\end{proof}

\section{The short part is spectral}
\label{sec:short-spectral}

The short-face contribution, although defined through embeddings, is in fact a spectral quantity, expressible through the non-backtracking matrix.

\begin{theorem}\label{thm:short-part-nonbacktracking-spectral}
For a connected cubic graph $G$ of girth $\ell$,
\[
S(G)=\sum_{k=\ell}^{2\ell-1}\frac{\tr(B^k)}{k\,2^{k}}.
\]
Moreover, if $G$ and $H$ are connected cubic graphs with the same adjacency spectrum, then $\ell(G)=\ell(H)$ and $S(G)=S(H)$. That is, for connected cubic graphs, the short-face contribution is determined by the adjacency spectrum.
\end{theorem}

\begin{proof}
By Definition~\ref{def:nonbacktracking}, $\tr(B^k)$ counts the closed non-backtracking walks of length $k$ with a distinguished starting dart. For $\ell\le k<2\ell$, such walks are exactly the traversals of simple cycles with a starting dart (the argument of Remark~\ref{rem:short-faces-simple-cycles}), so
\[
\tr(B^k)=2k\,c_k(G),\qquad\text{that is,}\qquad c_k(G)=\frac{\tr(B^k)}{2k}\quad(\ell\le k<2\ell).
\]
Substituting this into Corollary~\ref{cor:short-face-expectation}, we obtain
\[
S(G)=\sum_{k=\ell}^{2\ell-1}2^{1-k}c_k(G)=\sum_{k=\ell}^{2\ell-1}2^{1-k}\,\frac{\tr(B^k)}{2k}=\sum_{k=\ell}^{2\ell-1}\frac{\tr(B^k)}{k\,2^{k}}.
\]

Now, suppose $G$ and $H$ are connected cubic graphs with the same adjacency spectrum. The two graphs have the same $n$, and $m=\tfrac{3n}{2}$, so
by the Bass--Ihara determinant (Remark~\ref{rem:bass-ihara}) their
non-backtracking matrices have
$\tr(B_G^{\,k})=\tr(B_H^{\,k})$ for every $k$. Since $\tr(B^k)=2k\,c_k(G)$
for $k<2\ell$, the traces vanish below the girth and
$\tr(B^\ell)=2\ell\,c_\ell(G)>0$, so each graph's girth is
$\min\{k\ge1:\tr(B^k)>0\}$. Thus $\ell(G)=\ell(H)$, and the displayed formula gives $S(G)=S(H)$.
\end{proof}

\begin{remark}\label{rem:bass-ihara}
For a $d$-regular graph,
$\det(I-uB)=(1-u^2)^{m-n}\det\!\big(I-uA+(d-1)u^2I\big)$
\cite[Theorem~1.1, Corollary~1.2]{KotaniSunada}, originating with
Ihara~\cite{Ihara} and Bass~\cite{Bass}. Thus $n$, $m$, $d$, and the spectrum
of $A$ determine the spectrum of $B$.
\end{remark}

\begin{theorem}\label{thm:long-part-not-spectral}
Among connected cubic graphs, $L(G)$ is not determined by the adjacency spectrum: within each family of Section~\ref{sec:witnesses}, the members are cospectral and share $S(G)$, yet have pairwise distinct average genus, hence pairwise distinct $L(G)$.
\end{theorem}

\begin{proof}
By Theorem~\ref{thm:short-part-nonbacktracking-spectral}, $S(G)$ depends only on the adjacency spectrum, so cospectral cubic graphs have equal $S$. By Lemma~\ref{lem:genus-from-faces}, $\Exp[g]=1+\tfrac n4-\tfrac12\Exp[F]$, hence $\Exp[F]=2+\tfrac n2-2\Exp[g]$ and
\[
L(G)=\Exp[F]-S(G)=2+\tfrac n2-2\Exp[g]-S(G).
\]
Within either family of Section~\ref{sec:witnesses} the members share the order $n$ and, being adjacency-cospectral (Theorems~\ref{thm:n16-smallest-strong-witness} and~\ref{thm:main-witness-class}), share $S(G)$. Their average genera are pairwise distinct (Tables~\ref{tab:n16-smallest-witness} and~\ref{tab:n22-long-face-values}), so by the display their $L(G)$ are pairwise distinct. Thus, $L(G)$ is not determined by the spectrum.
\end{proof}

The six graphs of Figure~\ref{fig:n22-six-witnesses} have girth $\ell=3$ and $(c_3,c_4,c_5)=(1,1,3)$, so
\[
S(G)=2^{-2}c_3+2^{-3}c_4+2^{-4}c_5=\tfrac14+\tfrac18+\tfrac{3}{16}=\tfrac{9}{16}.
\]
Thus, the variation in their \emph{average genus} cannot come from short faces and must be from $L(G)$.

\begin{table}[htbp]
\centering
\begin{tabular}{c|c|c|c}
\toprule
graph & $\Exp[g]$ & $\Exp[F]$ & $L(G)$ \\
\midrule
$W_1$ & $4.899204$ & $3.201591$ & $2.639091$ \\
$W_2$ & $4.894855$ & $3.210289$ & $2.647789$ \\
$W_3$ & $4.889477$ & $3.221046$ & $2.658546$ \\
$W_4$ & $4.890793$ & $3.218414$ & $2.655914$ \\
$W_5$ & $4.891880$ & $3.216240$ & $2.653740$ \\
$W_6$ & $4.895828$ & $3.208344$ & $2.645844$ \\
\bottomrule
\end{tabular}
\caption{Average genus, expected face count, and long-face contribution for the six graphs of Figure~\ref{fig:n22-six-witnesses}.}
\label{tab:n22-long-face-values}
\end{table}

\section{Locally compatible circuits and the first non-simple difference}
\label{sec:longcircuits}

To explain the long-face contribution by length, we record, for each closed
non-backtracking circuit, when it can appear as a face of some orientable embedding. Recall that $\sigma_v(u)$ denotes the neighbor of $v$ following $u$ in the rotation at $v$ (Section~\ref{sec:preliminaries}).
As in the proof of Corollary~\ref{cor:short-face-expectation}, a facial walk
entering $v$ from $u$ and leaving toward $w$ requires exactly
$\sigma_v(u)=w$, and a single such constraint ($u\ne w$) is met
by exactly one of the two rotations at a cubic vertex. A walk can be realized as a
face only if all its constraints are met at once.

\begin{definition}\label{def:compatibility}
A \emph{circuit} of length $k$ is a cyclic sequence $C=(d_0,\dots,d_{k-1})$
of pairwise distinct darts, up to cyclic shift, in which the head of $d_i$ is
the tail of $d_{i+1}$ and $d_{i+1}\ne\alpha(d_i)$ (indices mod $k$). These are directed closed non-backtracking walks, not matroid circuits.
Each \emph{pass} $d_i=(u,v)$, $d_{i+1}=(v,w)$ imposes the constraint
$\sigma_v(u)=w$. Let $\mathcal C_k(G)$ be the set of circuits of
length $k$ and $\mathcal C_{\NB}(G)$ their union over $k$. Dart-simplicity of the circuits
gives $k\le2m$, so $\mathcal C_{\NB}(G)$ is finite. Define $\chi(C):=1$ if some
rotation system meets every constraint of $C$ ($C$ is \emph{locally
compatible}) and $\chi(C):=0$ otherwise, and let $r(C)$ be the number of
distinct vertices $C$ visits.
\end{definition}

\begin{theorem}\label{thm:compatible-circuit-face-expectation}
For a connected cubic graph $G$,
\[
\Exp[F]=\sum_{C\in\mathcal C_{\NB}(G)}\chi(C)\,2^{-r(C)},
\qquad
L(G)=\sum_{\substack{C\in\mathcal C_{\NB}(G)\\ |C|\ge 2\ell}}\chi(C)\,2^{-r(C)}.
\]
\end{theorem}

\begin{proof}
Each face of $\rho$ is a cycle of the permutation $\phi_\rho$, hence a circuit, so by linearity
\[
\Exp[F]=\sum_{C\in\mathcal C_{\NB}(G)}\Pr[C\text{ is a face of }\rho].
\]
A circuit $C$ is a face of $\rho$ iff every constraint of $C$ holds. The
forward direction is the discussion preceding
Definition~\ref{def:compatibility}, and conversely, if
$\sigma_v(u)=w$ at every pass, then $\phi_\rho(d_i)=d_{i+1}$ for
all $i$ and the pairwise distinct darts of $C$ form a single cycle of
$\phi_\rho$. At a cubic vertex a single constraint determines the rotation, so the constraints of $C$ at $v$ are met by exactly one of the two rotations or by neither. Because the rotations at the $r(C)$ visited
vertices are independent,
\[
\Pr[C\text{ is a face of }\rho]=\chi(C)\,2^{-r(C)}.
\]
For the second identity, the argument of
Remark~\ref{rem:short-faces-simple-cycles} makes the circuits of length
$k<2\ell$ oriented simple $k$-cycles, each with $\chi=1$ (one
satisfiable constraint per vertex) and $r=k$. Thus, they contribute
$\sum_{k=\ell}^{2\ell-1}2c_k(G)\,2^{-k}=S(G)$
(Corollary~\ref{cor:short-face-expectation}), and subtracting via
$L(G)=\Exp[F]-S(G)$ (Proposition~\ref{prop:short-long-decomposition}) gives
$L(G)=\sum_{|C|\ge2\ell}\chi(C)\,2^{-r(C)}$.
\end{proof}

We now separate genuine simple cycles from non-simple candidate facial walks.

\begin{definition}\label{def:cyc-nonsimp}
For $k\ge1$ let $A_k(G)=\sum_{C\in\mathcal C_k(G)}\chi(C)2^{-r(C)}$, split as $A_k=A_{k,\cyc}+A_{k,\nonsimp}$ according to whether $C$ is vertex-simple, and put $L_{\cyc}(G)=\sum_{k\ge2\ell}A_{k,\cyc}(G)$ and $L_{\nonsimp}(G)=\sum_{k\ge2\ell}A_{k,\nonsimp}(G)$.
\end{definition}

\noindent By the proof of Theorem~\ref{thm:compatible-circuit-face-expectation}, $A_k(G)=\Exp[F_{=k}]$ is the expected number of length-$k$ faces, and $L(G)=L_{\cyc}(G)+L_{\nonsimp}(G)$. We now determine, for the graph family in Figure~\ref{fig:n22-six-witnesses}, the first facial length at which the values $A_k(G)$ differ.

\begin{theorem}\label{thm:n22-first-compatible-difference}
For the six graphs of Figure~\ref{fig:n22-six-witnesses}, the values $A_k(G)$ coincide for all $3\le k\le 11$. The first length at which they differ is $k=12$. At length $12$, the simple-cycle contribution is still common,
\[
A_{12,\cyc}(G)=\tfrac{87}{2048},
\]
for all six graphs, and the entire difference lies in the non-simple contribution $A_{12,\nonsimp}(G)$.
\end{theorem}

\begin{proof}
The statement is a finite computation, carried out separately for each of the six graphs $G$ on $n=22$ vertices. The length-$12$ values are recorded in Table~\ref{tab:n22-first-compatible-difference}, and the enumeration in Appendix~\ref{app:data}.
\end{proof}

\begin{table}[H]
\centering
\begin{tabular}{c|c|c|c}
\toprule
graph & $A_{12}(G)$ & $A_{12,\cyc}(G)$ & $A_{12,\nonsimp}(G)$ \\
\midrule
$W_1$ & $135/2048$ & $87/2048$ & $3/128$ \\
$W_2$ & $151/2048$ & $87/2048$ & $1/32$ \\
$W_3$ & $151/2048$ & $87/2048$ & $1/32$ \\
$W_4$ & $151/2048$ & $87/2048$ & $1/32$ \\
$W_5$ & $135/2048$ & $87/2048$ & $3/128$ \\
$W_6$ & $135/2048$ & $87/2048$ & $3/128$ \\
\bottomrule
\end{tabular}
\caption{The first facial-length layer distinguishing the six $n=22$ witnesses.}
\label{tab:n22-first-compatible-difference}
\end{table}

At the first distinguishable length, the six graphs separate into two classes, corresponding to the two values of $A_{12}$ in Table~\ref{tab:n22-first-compatible-difference}. The remaining distinctions in $L(G)$ are carried by non-simple contributions at greater lengths.

\section{A genus bound from short cycles}
\label{sec:bounds}

The short-face identity also yields deterministic bounds. Because faces of length $<2\ell$ are simple cycles, short-cycle counts alone constrain how many faces any embedding can have, and hence bound the genus from below.
\begin{proposition}\label{thm:deterministic-short-cycle-bound}
Let $G$ be a connected cubic graph on $n$ vertices with girth $\ell$. For every
$\tau\in\{\ell-1,\ell,\ldots,2\ell-1\}$,
\[
\gamma(G) \ge
1+\frac n4-\frac12\left[
\frac{3n}{\tau+1}
+
\sum_{k=\ell}^{\tau}
\left(1-\frac{k}{\tau+1}\right)c_k(G)
\right],
\]
with the convention that the sum is empty for $\tau=\ell-1$. Since
$\gamma(G)$ is an integer, the bound also holds with the right side rounded up.
\end{proposition}

\begin{proof}
By Lemma~\ref{lem:genus-from-faces},
$\gamma(G)=1+\tfrac n4-\tfrac12\max_{\rho}F(\rho)$, so it suffices to bound
$F(\rho)$ for every $\rho$. Let $f_k$ be the number of length-$k$ faces.
Face lengths sum to $2m=3n$. For $k<2\ell$ every length-$k$ face traverses a
simple $k$-cycle (Remark~\ref{rem:short-faces-simple-cycles}), so $f_k=0$ for
$k<\ell$. Also, at most one face traverses a given cycle, since distinct faces
are disjoint cycles of $\phi_\rho$, while the two orientations would force
$\sigma_v(u)=w$ and $\sigma_v(w)=u$ at a common vertex, which no
rotation satisfies. Hence $f_k\le c_k(G)$ for $k<2\ell$. Every face of
length ${>}\tau$ has length at least $\tau+1$, so
\[
F(\rho)=\sum_{k=\ell}^{\tau}f_k+F_{>\tau}
\le\sum_{k=\ell}^{\tau}f_k+\frac{3n-\sum_{k=\ell}^{\tau}kf_k}{\tau+1}
=\frac{3n}{\tau+1}+\sum_{k=\ell}^{\tau}\Bigl(1-\frac{k}{\tau+1}\Bigr)f_k,
\]
and since each coefficient $1-\tfrac{k}{\tau+1}$ is positive, replacing
$f_k$ by $c_k(G)$ only increases the right side. Substituting the resulting
bound on $\max_\rho F(\rho)$ gives the claim. For $\tau=\ell-1$ the sums are
empty and the bound is $F(\rho)\le 3n/\ell$.
\end{proof}
At the base cutoff $\tau=\ell-1$, the proposition is the classical girth bound $\gamma(G)\ge 1+\tfrac n4-\tfrac{3n}{2\ell}$, which follows from Euler's formula (see, e.g., \cite{MoharThomassen}), and counting arguments of this kind are standard for bounding the genus of a fixed graph. The cutoffs $\tau\ge\ell$ improve it with cycle counts, and the strongest cutoff has a closed form. Writing $D(\tau)$ for the bound on $\max_\rho F(\rho)$ above, a direct computation gives
\[
D(\tau+1)-D(\tau)=\frac{\sum_{k=\ell}^{\tau+1}k\,c_k(G)-3n}{(\tau+1)(\tau+2)},
\]
so increasing the cutoff strengthens the bound exactly while $\sum_{k=\ell}^{\tau}k\,c_k(G)$ stays below $3n$. The bound is therefore strongest at the largest such $\tau$, and at the base cutoff when $\ell\,c_\ell(G)\ge 3n$.

\begin{example}[The bound on the Pappus graph]\label{ex:pappus}
The Pappus graph $P$ is a connected cubic graph on $n=18$ vertices with girth $\ell=6$ and short-cycle counts
\[
c_6=18,\quad c_7=0,\quad c_8=54,\quad c_9=0,\quad c_{10}=54,\quad c_{11}=0 .
\]
Here, $\ell\,c_\ell=6\cdot 18=108>54=3n$, so the base cutoff is strongest, and the bound is the girth bound
\[
\gamma(P)\ \ge\ 1+\frac{18}{4}-\frac{3\cdot 18}{2\cdot 6}\ =\ 1.
\]
It is attained, since the Pappus graph is toroidal and so $\gamma(P)=1$. Already at $\tau=6$ the over-count $f_6\le c_6$ is far from tight and the bound is only $\tfrac{5}{14}\approx 0.357$, which still rounds up to $\gamma(P)\ge1$. In particular, the bound is not monotone in the cutoff $\tau$.
\end{example}

We tested the deterministic bound across the connected cubic census through $20$ vertices, taking for each graph the strongest value over all cutoffs $\ell\le\tau<2\ell$ (Table~\ref{tab:deterministic-short-cycle-bound-performance}). Through $20$ vertices, the bound is always within $2$ of the true genus, exact for $26.9\%$ of the $510{,}489$ graphs on $20$ vertices and within $1$ for $89.6\%$.

\begin{table}[htbp]
\centering
\begin{tabular}{c|c|c|c|c|c|c}
\toprule
$n$ & $\#G$ & mean $\gamma$ & mean gap & median gap & max gap & \% tight \\
\midrule
$8$  & $5$      & $0.400$ & $0.000$ & $0$ & $0$ & $100.0$ \\
$10$ & $19$     & $0.526$ & $0.263$ & $0$ & $1$ & $73.7$ \\
$12$ & $85$     & $0.624$ & $0.306$ & $0$ & $1$ & $69.4$ \\
$14$ & $509$    & $0.758$ & $0.460$ & $0$ & $2$ & $56.0$ \\
$16$ & $4060$   & $0.884$ & $0.559$ & $1$ & $2$ & $47.6$ \\
$18$ & $41301$  & $1.034$ & $0.682$ & $1$ & $2$ & $38.2$ \\
$20$ & $510489$ & $1.213$ & $0.835$ & $1$ & $2$ & $26.9$ \\
\bottomrule
\end{tabular}
\caption{Performance of the deterministic short-cycle lower bound through $20$ vertices.}
\label{tab:deterministic-short-cycle-bound-performance}
\end{table}

\section{Synthesis: between spectrum and matroid}
\label{sec:synthesis}

We have seen that when the mean face count is split as $\Exp[F]=S(G)+L(G)$, the short part $S(G)$ is determined by spectral data (Theorem~\ref{thm:short-part-nonbacktracking-spectral}). Also, the entire polynomial $\Gamma_G$ is determined (for cubic $G$) by the cycle matroid (Theorem~\ref{thm:matroid-invariance}). Further, the witnesses of Section~\ref{sec:witnesses} show that the polynomial carries information the spectrum does not (Theorems~\ref{thm:n16-smallest-strong-witness} and~\ref{thm:main-witness-class}). At the level of the expected face count, then, the spectrum determines the short term $S(G)$, and the part it misses is the long term $L(G)$, whose first difference Section~\ref{sec:longcircuits} identifies as locally compatible non-simple facial structure. The infinite family of Theorem~\ref{thmD} gives one boundary case: its pairs are separated at the minimum genus by a Godsil--McKay switch that necessarily changes the cycle matroid (Remark~\ref{rem:gm-matroid}).

\begin{remark}[The role of connectivity]\label{rem:connectivity}
The cycle matroid result behaves differently depending on connectivity. For $2$-connected graphs that are not $3$-connected, non-isomorphic cubic graphs can share a cycle matroid. For $3$-connected graphs, Whitney's theorem shows that the cycle matroid determines the graph, so the matroid bound reduces to the statement that the genus polynomial is a graph invariant.
\end{remark}

\begin{remark}[Incompleteness]\label{rem:incompleteness}
The census shows first that the genus polynomial does not determine the \emph{graph}. That it does not determine the \emph{cycle matroid} is strictly stronger. Example~\ref{ex:converse} supplies one among \emph{$3$-connected} cubic graphs, where Whitney's theorem makes the cycle matroid determine the graph, so there, non-isomorphic graphs force non-isomorphic matroids. Thus, the implication ``$M(G)\cong M(H)\Rightarrow\Gamma_G=\Gamma_H$'' of Theorem~\ref{thmA} is proper.
\end{remark}

\begin{example}[A converse witness]\label{ex:converse}
A single pair demonstrates both Remark~\ref{rem:incompleteness} and Theorem~B$'$. The two $3$-connected cubic graphs on $12$ vertices with \texttt{graph6} codes \verb|K??FEbGL@WB_| and \verb|K?ABE`Ke@WEO| share the orientable genus polynomial
\[
\Gamma(x)=64x+1408x^{2}+2624x^{3},
\]
yet are non-isomorphic: they have $7938$ and $8100$ spanning trees, respectively. The unequal spanning-tree counts prove directly that the cycle matroids are non-isomorphic, so the genus polynomial does not determine the cycle matroid even among $3$-connected cubic graphs (Remark~\ref{rem:incompleteness}). Since the spanning-tree count of a cubic graph is determined by its adjacency spectrum, the two graphs also have distinct adjacency spectra. Twelve is the smallest order at which two non-isomorphic connected cubic graphs share a genus polynomial. 
\end{example}
  \section{Conclusion}
\label{sec:questions}

We now return to the question from the introduction of how strong an
invariant the genus polynomial of a cubic graph is. We have seen the
answer is two-sided. It is not determined by the adjacency spectrum and,
for finite simple $2$-connected cubic graphs, it is determined
by the cycle matroid. We now note some remaining questions. Theorem~\ref{thmA} needs the cubic hypothesis, and
Remark~\ref{rem:cubic-necessary} shows it cannot be dropped. One
can naturally ask what lies between: outside of cubic graphs, for which
classes is the genus polynomial a cycle-matroid invariant? How often does the genus polynomial distinguish connected cubic graphs,
asymptotically? The distinct fraction $p_n/g_n$ of
Table~\ref{tab:census} first decreases and then increases over the
computed range, staying between $70\%$ and $100\%$. Does it converge, and if so
to what limit? Finally, Theorem~\ref{thmC} expresses the short part $S(G)$ through the
non-backtracking matrix, and we do not know an analogue for the long
part. Is $L(G)$, or its compatible-circuit refinement, expressible
through an algebraic object similar to the non-backtracking matrix? By
Theorem~\ref{thm:long-part-not-spectral}, any such operator cannot be a
function of the adjacency matrix alone.

\appendix

\section{Computational methods and data}
\label{app:data}

\subsection{Reproducibility}
We computed every genus polynomial by enumerating all $2^n$ rotation systems, tracing the face permutation $\phi_\rho=\sigma_\rho\alpha$, and reading the genus from the face count by Lemma~\ref{lem:genus-from-faces}. The enumeration is exhaustive because the polynomial counts embeddings at every genus, not just the least. Two independent implementations agree on every printed polynomial.

\subsection{Supplementary data}
The scripts, run logs, and certificate files for the computations below are
archived at \url{https://doi.org/10.5281/zenodo.21639035}. The per-graph census
output is too large to deposit and is regenerated by those scripts. The small
certificates are recorded here.
\begin{enumerate}
\item[\textup{(D1)}] The census of all connected cubic graphs through $22$ vertices, generated by \texttt{geng}~\cite{McKayPiperno}, each tagged with its genus polynomial and an exact non-real-root test. The counts are those of Table~\ref{tab:census}.
\item[\textup{(D2)}] The three graphs of Table~\ref{tab:n16-smallest-witness} and the six graphs of Figure~\ref{fig:n22-six-witnesses}, with the shared invariants listed in Theorems~\ref{thm:n16-smallest-strong-witness} and~\ref{thm:main-witness-class}, and the search over orders $8$ to $16$ showing no such triple occurs earlier.
\item[\textup{(D3)}] For each of the eight single-switch pairs, a four-element set meeting Definition~\ref{def:gm-switch} and the labeling that realizes the switch. For the other ten pairs we give the completed negative searches over sizes $2,4,6$.
\item[\textup{(D4)}] The rotation systems realizing the two embeddings of Lemma~\ref{lem:if-base}, used in the induction of Lemma~\ref{lem:if-upper}. We verified $\gamma(G_1)=2$ and $\gamma(H_1)=1$ with PAGE~\cite{MetzgerUlrigg}.
\item[\textup{(D5)}] For the six graphs of Figure~\ref{fig:n22-six-witnesses}, the expected face counts by length, split into simple and non-simple. The six share every simple-cycle count, so their first difference, at length $12$, is entirely non-simple.
\item[\textup{(D6)}] The two graphs of Example~\ref{ex:converse}, with codes \verb|K??FEbGL@WB_| and \verb|K?ABE`Ke@WEO|, sharing $64x+1408x^2+2624x^3$ but with $7938$ and $8100$ spanning trees, and the script that found them.
\end{enumerate}

\end{document}